\documentclass[11pt, twoside, leqno]{article}
\usepackage{graphicx} % Required for inserting images
\usepackage{amsfonts}
\usepackage{mathtools}
\usepackage{amssymb}
\usepackage{amscd}
\usepackage{amsthm}
\usepackage[all,cmtip]{xy}
\usepackage{hyperref}
\usepackage{yhmath}
\usepackage{amsmath}
\usepackage{color}
\usepackage{mathrsfs}
\usepackage{txfonts}
\usepackage{amsfonts}

\usepackage{indentfirst}
\usepackage{ulem}
\usepackage[all]{xy}
\usepackage{mathdots}
\usepackage{leftidx}
\usepackage{mathrsfs}
\usepackage{amsmath,amssymb, amsthm}
\usepackage{color}
\usepackage{tikz}
\usepackage{picture}
\usepackage{enumerate}
\usepackage{makecell}
\usepackage{extarrows}
\usepackage{graphicx}
\usepackage{caption}
\usepackage{subcaption}
\usepackage[pagewise]{lineno}
\allowdisplaybreaks

\numberwithin{equation}{section}
\newtheorem{Thm}{Theorem}[section]
\newtheorem{Lem}[Thm]{Lemma}

\newtheorem{Prop}[Thm]{Proposition}
\newtheorem{Ex}[Thm]{Example}
\newtheorem{Rem}[Thm]{Remark}
\newtheorem{Conj}[Thm]{Conjecture}

\newcommand{\RNum}[1]{\uppercase\expandafter{\romannumeral #1\relax}}
\newcommand{\cat}[1]{\mathscr{#1}}
\renewcommand{\appendix}{\par
\setcounter{section}{0}%
\setcounter{subsection}{0}%
\setcounter{subsubsection}{0}%
\gdef\thesection{\@Alph\c@section}%
\gdef\thesubsection{\@Alph\c@section.\@arabic\c@subsection}%
\gdef\theHsection{\@Alph\c@section.}%
\gdef\theHsubsection{\@Alph\c@section.\@arabic\c@subsection}%
\csname appendixmore\endcsname
}
\title{\bf\Large Cohomology of classifying spaces of rank 3 Kac-Moody groups
\footnotetext{\hspace{-0.35cm} 2020 {\it
Mathematics Subject Classification}. Primary 55N45.
\endgraf {\it Key words and phrases.}
Cartan Matrix; Kac-Moody group; Weyl group; Cohomology group; Classifying space; Invariants.
\endgraf This project is supported by the National
Natural Science Foundation of China, 11571038.}}
\author{Yangyang Ruan and Xu-an Zhao}
\date{}
\numberwithin{equation}{section}
\title{Cohomology of classifying spaces of rank 3 Kac-Moody groups}

\begin{document}
\arraycolsep=1pt

\maketitle

\vspace{-0.3cm}

\begin{center}
\begin{minipage}{13cm}
{\small {\bf Abstract}\quad
We represent the rational and mod $p$ cohomology groups of classifying spaces of rank 3 Kac-Moody groups by a direct sum of the invariants of Weyl groups and their quotients. As an application, we conclude that there is a $p$-torsion for each prime $p$ in the integral cohomology groups of classifying spaces of rank 3 Kac-Moody groups. We also determine the ring structure of the rational cohomology with one exception case. }
\end{minipage}
\end{center}

\section{Introduction\label{s1}}

\subsection{Background}
The goal of this paper is to compute the rational and mod $p$ cohomology groups of the classifying spaces of rank 3 Kac-Moody groups. A Kac-Moody group generalizes a Lie group, as its associated Lie algebra can be infinite-dimensional. From the perspective of homotopy theory, the natural object associated with a Kac-Moody group is its classifying space. Nitu Kitchloo was the first to study the classifying spaces of Kac-Moody groups in his thesis \cite{Kit98}. He discovered a homotopy equivalence between the classifying space of a Kac-Moody group and the homotopy colimit of the classifying spaces of its parabolic subgroups of finite type (see \cite[Theorem 4.2.4]{Kit98} or Theorem \ref{s2t2} for details). This result allows the classifying space of a Kac-Moody group to be constructed homotopically from the classifying spaces of its parabolic subgroups, providing a framework for computing cohomology.

Several subsequent works \cite{Kit14, Kit17, BK02, ABKS05} have built on this foundation. Notably, Aguadé, Broto, Kitchloo, and Saunell \cite{ABKS05} computed the mod $p$ cohomology algebra of the classifying spaces of central quotients of rank 2 Kac-Moody groups. Foley has also made significant contributions to the study of these classifying spaces \cite{Fol12, Fol15, Fol22}.

However, little is known about the cohomology of classifying spaces of Kac-Moody groups of indefinite type. Beyond the rank 2 case, there are virtually no examples of computed cohomology for their classifying spaces. Motivated by this gap, our work extends the study to the rank 3 case. Our results on the structure of the classifying spaces of Kac-Moogy groups show a striking difference with those for Lie groups. For example, their integral cohomology groups exhibit
$p$-torsion for every prime $p$, in stark contrast to classical Lie groups, whose integral cohomology groups contain torsion only at $p=2,3,5$.

\subsection{Main Result}
In this paper, we focus on simply connected Kac-Moody groups. Using Kitchloo's result \cite[Theorem 4.2.4]{Kit98}, we classify rank 3 Kac-Moody groups of infinite type in \ref{s3ht}, dividing them into four distinct classes. This classification corresponds directly to the classification of rank 3 Cartan matrices of infinite type, due to the one-to-one correspondence between Cartan matrices $A$ and Kac-Moody groups $G(A)$.

We use the Mayer-Vietoris sequence to compute the rational cohomology groups $H^*(BG(A);\,\mathbb{Q})$ of the classifying space $BG(A)$ for each class in Proposition \ref{s3ht}. The input involves the rational cohomology of the classifying spaces of parabolic subgroups of finite type. By Borel's theorem \cite{Bor53}(see also Theorem \ref{s2t1}), their rational cohomology is isomorphic to the invariants of their Weyl groups on $H^*(BT;\,\mathbb{Q})$. For a Kac-Moody group $G(A)$, we denote its parabolic subgroup indexed by $J$ as $G_{J}(A)$. Notably, $\emptyset$ refers to the maximal torus $G_{\emptyset}(A)=T$. For simplicity, in the rank 3 case, we omit brackets and commas when referring to indices, e.g., $G_{13}(A)$ represents $G_{\{1,~3\}}(A)$, corresponding to the simple roots $\{\alpha_{1},~\alpha_{3}\}$. Similarly, $W_{J}(A)$  denotes the subgroup of $W(A)$ generated by elements indexed by $J$. By applying Borel's result \cite{Bor53}, we obtain:
 $$H^*(BG_{J}(A);\,\mathbb{Q})\cong H^{*}(BT;\,\mathbb{Q})^{W_{J}(A)}.$$

Let $P=H^{*}(BT;\,\mathbb{Q})$ and $P_J=P^{W_{J}(A)}$. Clearly, $P_{\emptyset}=P$, and for $H\subseteq J$, $P_J\subseteq P_H$.
\begin{Thm}[Theorem \ref{s3t1}]\label{RCHG}
The rational cohomology groups $H^*(BG(A);\,\mathbb{Q})$ corresponding to the four classes of rank 3 Kac-Moody groups of infinite type in Proposition \ref{s3ht} are represented as follows:
\begin{enumerate}
\item[{\rm (i)}] $\Sigma(P-(P_{1}+P_{2}))\bigoplus \Sigma(P-(P_{12}+P_{3}))\bigoplus P_{123};$
\item[{\rm (ii)}] $\Sigma(P-(P_{12}+P_{3}))\bigoplus P_{123};$
\item[{\rm (iii)}] $\Sigma(P_{1}-(P_{12}+P_{13}))\bigoplus P_{123};$
\item[{\rm (iv)}] $\Sigma^{2}(P-(P_{1}+P_{2}+P_{3}))\bigoplus \Sigma(P_{1}\cap (P_{2}+P_{3})-(P_{12}+P_{13}))\bigoplus P_{123},$
\end{enumerate}
\end{Thm}
where the minus symbol `$-$' denotes the quotient of invariants, and
$\Sigma$ represents a shift in cohomology degree by one.

The computations of rational and mod $p$ cohomology groups are largely similar, except for the consideration of $p$-torsions in the integral cohomology groups of parabolic subgroups. When these groups are $p$-torsion free, the mod $p$ cohomology has the same form as the rational case, replacing $\mathbb{Q}$ with  $\mathbb{F}_{p}$. Otherwise, the computations become more complex.

All rank 1 Kac-Moody groups are of finite type, and the integral cohomology groups of their classifying spaces are torsion-free. For rank 2 simply connected Kac-Moody groups of finite type, their integral cohomology groups are $p$-torsion free for all primes except
$p=2$, as shown in \cite[Theorem 5.11]{MT91}. This exception arises because the integral homology group of the Lie group $G_2$ contains 2-torsion, and $G_2$ can appear as a parabolic subgroup of rank 3 Kac-Moody groups. Consequently, we divide the mod $p$ cohomology computations into two cases: $p>2$ and $p=2$.

For mod $p$ ($p>2$) cohomology, they have the same form Theorem \ref{s3t2} as the rational case;
For mod 2 cohomology, we refine the classification into ten classes based on the appearance of $G_2$ as a rank 2 finite parabolic subgroup. Their cohomology groups are presented in Theorem \ref{s4t1}.

Although the cohomology groups of the classifying spaces of rank 3 Kac-Moody groups of infinite type, as presented in Theorems \ref{RCHG}, \ref{s3t2}, and \ref{s4t1}, do not follow the standard form of abelian groups, they still provide significant insights into the integral cohomology groups of these classifying spaces. A direct corollary of these results generalizes Borel's Theorem \cite{Bor53} from Lie groups to rank 3 Kac-Moody groups. Specifically, the image of the cohomology homomorphism
$\rho^{*}_{\mathbb{F}}:H^{*}(BG(A);\,\mathbb{F})\rightarrow H^{*}(BT;\,\mathbb{F})$ , induced by the map $\rho: BT \rightarrow BK(A)$, is given by $H^{*}(BT;\,\mathbb{F})^{W(A)}$ when $\mathbb{F}$ is $\mathbb{Q}$ or $\mathbb{F}_{p}$. By combining this with a computation of invariants, we arrive at the following conclusion:
\begin{Thm}[Theorem \ref{s6t1}]
There is a $p$-torsion  for each prime $p$ in the integral cohomology groups of classifying spaces of rank 3 Kac-Moody groups of infinite type.
\end{Thm}
We have demonstrated that $H^*(BG(A);\,\mathbb{Q})$ can be expressed as a direct sum of invariants of Weyl groups and their quotients, which provides the group structure. Additionally, we establish a vanishing result (Theorem \ref{s8t1}) for the first direct summand of the rational cohomology in case {\rm (i)} of Theorem \ref{RCHG}. This result enables us to determine the ring structure of $H^*(BG(A);\,\mathbb{Q})$ for the first three cases in Theorem \ref{RCHG}; further details can be found in Theorem \ref{RCR}.

Our paper is \textbf{organized} as follows. In Section 2, we recall key results about Kac-Moody groups and introduce the notations and tools used in our computations. In Section 3, we compute the rational and mod $p$ ($p>2$) cohomology groups of their classifying spaces. Section 4 presents a refined classification of rank 3 Kac-Moody groups of infinite type and computes the mod 2 cohomology groups of their classifying spaces. In Section 5, to highlight the differences in
$H^*(BG(A))$ under various coefficients, we provide an example of a specific Kac-Moody group and compute its rational, mod 3, and mod 2 Poincar é series. In Section 6, as an application of our computations, we prove that there is a $p$-torsion for each prime $p$ in the integral cohomology groups of the classifying spaces of rank 3 Kac-Moody groups of infinite type. Finally, in Section 7, we determine the ring structure of the rational cohomology groups of the classifying spaces of rank 3 Kac-Moody groups of infinite type for all cases except the last one.

\section{Preliminaries \label{s2}}
In this section, we recall some important results about Kac-Moody groups and introduce some notations for the computation of cohomology groups of their classifying spaces.

\subsection{Kac-Moody groups \label{s2.1} }
A good reference for Kac-Moody groups is \cite{Kac90}.

Let $A=(a_{ij})_{n\times n}$ be an $n\times n$ integer matrix; $A$ is called a (\textit{generalised}) \textit{Cartan matrix} if it satisfies the following conditions:\\
(1) For each $i$, $a_{ii}=2$;\\
(2) For $i\neq j$, $a_{ij}\leq 0$;\\
(3) For $i\neq j$, if $a_{ij}=0$, then $a_{ji}=0$.\\
For each Cartan matrix $A$, there are a \textit{Kac-Moody algebra} $g(A)$ and a simply connected \textit{Kac-Moody group} $G(A)$. And $n$ is called the \textit{rank} of $g(A)$ or $G(A)$.

A Cartan matrix $A$ is \textit{symmetrizable} if there exists a real invertible diagonal matrix $D$ and a symmetric matrix $B$ such that $A=DB$. Kac-Moody group $G(A)$ and Kac-Moody algebra $g(A)$ are \textit{symmetrizable} if the associated Cartan matrix $A$ is symmetrizable. A Cartan matrix $A$ is called \textit{indecomposable} if $A$ cannot be decomposed into a direct sum $A_{1}\oplus A_{2}$ of two Cartan matrices $A_{1}$ and $A_{2}$. Kac-Moody group $G(A)$ is \textit{indecomposable} if the associated Cartan matrix $A$ is indecomposable.

The indecomposable Cartan matrices or their associated Kac-Moody algebras or their associated Kac-Moody groups are divided into three types:\\
(1) Finite type, when $A$ is positive definite;\\
(2) Affine type, when $A$ is positive semidefinite and ${\rm rank}(A)=n-1$;\\
(3) Indefinite type, otherwise.

An indecomposable Cartan matrix $A$ is called \textit{infinite type} if it is of affine type or indefinite type. If $A$ is of finite type, then its associated Kac-Moody group $G(A)$ is the Lie group. To save space, instead of stating Cartan matrix or Kac-Moody group of finite (infinite) type, we will simply refer to them as the finite (infinite) Cartan matrix or Kac-Moody group, respectively. Obviously, a rank 2 Kac-Moody group $G(A)$ is infinite type if and only if $a_{12}a_{21}\geq 4$.

The Kac-Moody group $G(A)$ supports a canonical antilinear involution $\omega$, and one defines the unitary form $K(A)$ as the fixed group $G(A)^{\omega}$. As the inclusion map $K(A)\hookrightarrow G(A)$ is a homotopy equivalence, we work with $K(A)$ instead of $G(A)$ for convenience. Let $I$ be the set $\{1,~2,~\cdots,~n\}$. A subset
$J\subseteq I$ gives rise to a parabolic subalgebra $g_{J}(A)\subseteq g(A)$.
One may exponentiate these subalgebras to parabolic subgroups $G_{J}(A)\subseteq G(A)$. Let the unitary
Levi factors $K_{J}(A)$ be the groups $K(A)\cap G_{J}(A)$. Let $N(T)$ be the normalizer of $T$ in $K(A)$.
The \textit{Weyl group} $W(A)$ of $K(A)$ is the quotient group $N(T)/T$.
$W(A)$ is a Coxeter
group generated by Weyl reflections $\sigma_{i}$, $i\in I$. It has a Coxeter presentation given as follows:
$$W(A)=\left\langle~\sigma_{i} \mid ~\sigma^{2}_{i}=1,~i\in I;~(\sigma_{k}\sigma_{j})^{m_{kj}}=1,~ 1\leq k < j \leq n ~\right\rangle,$$
where $m_{kj}$ depends on the Cartan matrix $A=(a_{ij})_{n\times n}$, and
\begin{center}
\begin{tabular}{|c|c|c|c|c|c|}
\hline $a_{kj}a_{jk}$&0&1&2&3&$\geq$ 4 \\
\hline $m_{kj}$ &2&3&4&6&$\infty$ \\
\hline
\end{tabular}.
\end{center}For $J\subseteq I$, let $W_{J}(A)$ be the subgroup of $W(A)$ generated by the corresponding reflections $\sigma_{j}$, $j\in J$. The group $W_{J}(A)$ is also a Coxeter group that can be identified with the Weyl group of $K_{J}(A)$.

For the action of the Weyl group, a good reference is \cite{MT91}. The action of $W(A)$ on $T$ is defined by
$W(A)\times T\rightarrow T, ([s], t)\mapsto s^{-1}ts$, for $[s]\in W(A), t\in T$.
The induced action on the classifying space $BT=EK(A)/T$ of $T$ is $W(A)\times BT\rightarrow BT, ([s], bT)\mapsto bsT$, for $[s]\in W(A), bT \in BT$. This action on $BT$ induces an action on $H^{*}(BT;R)$ with a specific coefficient ring $R$. Then the cohomology ring of $H^{*}(BT; R)$
can be represented by the polynomial ring $R[\omega_{1},~\cdots,~\omega_{n}]$ with  the degree $\deg\omega_{i}=2$ of the fundamental dominant weight for any $i\in I$. Let $I_{n}$ be the $n\times n$ identity matrix, $A_{j}$ denote the matrix in which we use zeros to replace all elements of Cartan matrix $A$ except the $j$-th column, and $\alpha_{j}=\sum\limits^{n}_{i=1}a_{ij}\omega_{i}$ for any $j\in I$. Then, the induced action by generators $\{~\sigma_{j}~|~j\in I~\}$ of $W(A)$ can be written as follows:
\begin{align*}
W(A)\times H^{*}(BT;R)&\longrightarrow H^{*}(BT;R)\\
\sigma_{j}(~\omega_{1},~\cdots,~\omega_{n}~)&=(~\omega_{1},~\cdots,~\omega_{n}~)(~I_{n}-A_{j}~)\\
&=(~\omega_{1},~\cdots,~\omega_{j-1},~\omega_{j}-\alpha_{j},~\omega_{j+1},~\cdots,~\omega_{n}~).
\end{align*}

\subsection{A fundamental result about infinite Kac-Moody groups}
For a rank $n$ Cartan matrix $A$, let $\cat P(A)$ denote the category with objects that are proper subsets $J$ of $I$ with $K_J(A)$ of finite type, and morphisms given by inclusions of these subsets.
Now we recall that a fundamental result about infinite Kac-Moody groups is due to Kitchloo.
\begin{Thm}(\cite[Theorem 4.2.4]{Kit98}) \label{s2t2}
Let $K(A)$ be an infinite Kac-Moody group, then there is a homotopy equivalence:
$$BK(A)\simeq \mathop{\rm hocolim}_{J \in \cat P(A) } BK_{J}(A).$$
\end{Thm}
To simplify the homotopy colimit, we first construct a covariant functor from $\cat P(A)$ to the category $\cat Top$ of topology spaces:
\begin{align*}
\textbf{F}: \cat P(A) &\rightarrow \cat Top \\
J  &\mapsto BK_{J}(A)\\
H\subseteq J &\mapsto BK_{H}(A)\rightarrow BK_{J}(A).
\end{align*}
And we modify $\textbf{F}$ to $\textbf{F}'$, which sends the object $J$ to a homotopy replacement $X_J$ of $BK_J(A)$, and sends the morphism $H\subseteq J$ of sets to a cofibration map $X_H\subseteq X_J$ between topology spaces. See \cite{ZGR24} for more details. Then for any two subsets $H, J \subseteq I$, we have
$$X_{H\cap J}=X_{H}\cap X_{J}.$$
Let $P(A)$ be the set $\{J\in \cat P(A)\mid J\text{ is maximal under the inclusion of subsets}\}$.
Combining the homotopy replacement with Theorem \ref{s2t2}, we have
\begin{equation}\label{s2f1}
BK(A)\simeq \mathop{\rm hocolim}_{J \in \cat P(A)} BK_{J}(A)\simeq \mathop{\rm hocolim}_{J \in \cat P(A)} X_{J} \simeq \bigcup_{J \in \cat P(A)}X_{J}=\bigcup_{J \in P(A)}X_{J},
\end{equation}
which reduces the index set ${\cat P(A)}$ of the homotopy colimit to the set $P(A)$.

%We classify for rank 3 infinite Cartan matrices and rank these classes by Roman numerals according to the number of rank 2 finite parabolic subgroups in the next section. The pasting information of $BK(A)$ refers to the type of all parabolic groups of $K(A)$.

Here we use an example to illustrate these abstract symbols, and we further introduce a simplicial representation for $\cat P(A)$.

\begin{Ex}\label{s2e1}
Given a Cartan matrix $$A=\left(
\begin{array}{ccc}
2 & -1 & -3\\
-3 & 2 & -1\\
-2 & -4 & 2
\end{array}
\right),$$ then its all $2\times 2$ principal submatrices are $$\left(
\begin{array}{cc}
2 & -1 \\
-3 & 2
\end{array}
\right),~~~ \left(
\begin{array}{cc}
2 & -3\\
-2 & 2
\end{array}
\right),~~~ \left(
\begin{array}{cc}
2 & -1\\
-4 & 2
\end{array}
\right).$$ The first submatrix is the only one of finite type, hence $A$ is indefinite. Obviously, all rank 1 parabolic subgroups of $K(A)$ are of finite type.

The objects of $\cat P(A)$ are $\emptyset,\,\{ 1\},\,\{2\},\,\{3\},\,\{1, 2\}$, so $P(A)$ is $\{\{3\},\,\{1, 2\}\}$. In other words,

\begin{itemize}
\item $K_{\emptyset}(A),\,K_{1}(A),\,K_{2}(A),\,K_{3}(A),\,K_{12}(A)$ are finite parabolic subgroups,
\item $K_{13}(A),\,K_{23}(A)$ and $\,K_{123}(A)=K(A)$ are infinite parabolic subgroups,
\item $\,K_{3}(A),\,K_{12}(A)$ are maximal finite parabolic subgroups.
\end{itemize}
For convenience, we use the following simplicial complex to represent $\cat P(A)$.
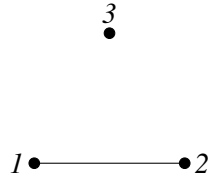
\begin{center}
\bigskip
\begin{tikzpicture}
\filldraw[black] (0,0) circle (2pt) node[anchor=east]{1}
(2,0) circle (2pt) node[anchor=west]{2}
(1,1.73) circle (2pt)node[anchor=south]{3};
\draw (0,0)--(2,0);
\end{tikzpicture}
\end{center}

In fact, we may view $\cat P(A)$ as a \textbf{simplicial} \textbf{complex} and $P(A)$ as the set of its  \textbf{maximal}  \textbf{simplices}.
%$For the simplicial representation, a given simplicial complex represents a class of infinite Kac-Moody groups. Those faces that are in this complex represent their parabolic subgroups of finite type, and those simplices that are not in this complex represent their parabolic subgroups of infinite type.

Furthermore, by Formula $\ref{s2f1}$, we get $BK(A)\simeq X_{12}\bigcup_{X_{\emptyset}} X_{3}$.
\end{Ex}

\subsection{A classification for rank 3 infinite
Cartan matrices}

For a given infinite Cartan matrix $A$, the set of objects of $\cat P(A)$ is not the only one by which we can obtain the homotopy type of $BK(A)$, but $P(A)$ is the minimal one among these sets. In fact, $\cat P(A)$ and $P(A)$ are equivalent in the sense that we can recover the objects of $\cat P(A)$ by adding all subsets of each element in $P(A)$.
To ignore the interference of indicators, we may define an equivalence relation on the set of rank $n$ Cartan matrices:  $A$ is equivalent to $B$ if there exists a permutation
$\sigma \in S_n$ such that $\sigma P(A)=P(B)$. Then we chose $P(A)$ as the representative element of each equivalent class and hence give a classification for rank $n$ infinite Cartan matrices up to symmetry of indices.

The simplicial representation of ${\cat P}(A)$ or $P(A)$ in Example \ref{s2e1} can be generalized to any rank Cartan matrix $A$. Let $|J|$ denote the number of elements in $J$. We could relate an object $J\in P(A)$ to a proper $|J|-1$ dimensional face of the $n-1$ dimensional simplex $I$ and we still use $J$ to denote a proper face that is also a $|J|-1$ dimensional simplex. Therefore $P(A)$ is viewed as a simplicial  complex.

For rank 3 infinite Kac-Moody groups, there are eight different forms of $P(A)$
\begin{align*}
&\left\{\{1\},~\{2\},~\{3\}\right\},&\left\{\{1,2\},~\{3\}\right\},&~~~~~~~~~\left\{\{1,3\},~\{2\}\right\},&\left\{\{2,3\},~\{1\}\right\},\\
&\left\{\{1,2\},~\{1,3\}\right\},&\left\{\{1,2\},~\{2,3\}\right\},&~~~~~~\left\{\{1,3\},~\{2,3\}\right\},&\left\{\{1,2\},~\{1,3\},~\{2,3\}\right\}.
\end{align*}
Up to symmetry of indices \{1, 2, 3\}, there are essentially only four classes. In what follows, we list these classes and their simplicial representations. We rank these classes by Roman numerals according to the number of rank 2 finite parabolic subgroups.
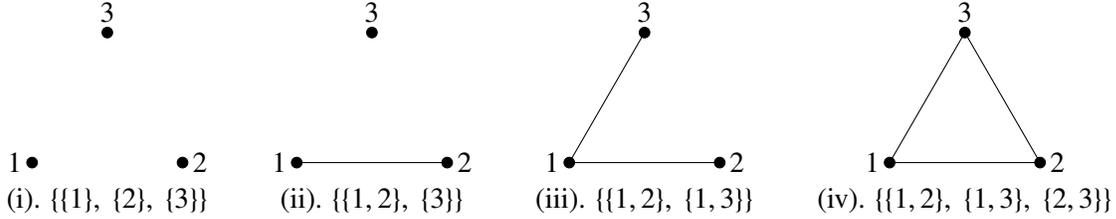
\begin{center}
\begin{tikzpicture}
\filldraw[black] (0,0) circle (2pt) node[anchor=east]{1}
(2,0) circle (2pt) node[anchor=west]{2}
(1,1.73) circle (2pt)node[anchor=south]{3};
\draw (1,-0.5) node{\rm (i). $\{\{1\},~\{2\},~\{3\}\}$};
\end{tikzpicture} \hspace{0.4cm}
\begin{tikzpicture}
\filldraw[black] (0,0) circle (2pt) node[anchor=east]{1}
(2,0) circle (2pt) node[anchor=west]{2}
(1,1.73) circle (2pt)node[anchor=south]{3};
\draw (1,-0.5) node{\rm (ii). $\{\{1,2\},~\{3\}\}$};
\draw (0,0)--(2,0);
\end{tikzpicture} \hspace{0.4cm}
\begin{tikzpicture}
\filldraw[black] (0,0) circle (2pt) node[anchor=east]{1}
(2,0) circle (2pt) node[anchor=west]{2}
(1,1.73) circle (2pt)node[anchor=south]{3};
\draw (1,-0.5) node{\rm (iii). $\{\{1,2\},~\{1,3\}\}$};
\draw (1,1.73)-- (0,0)--(2,0);
\end{tikzpicture} \hspace{0.4cm}
\begin{tikzpicture}
\filldraw[black] (0,0) circle (2pt) node[anchor=east]{1}
(2,0) circle (2pt) node[anchor=west]{2}
(1,1.73) circle (2pt)node[anchor=south]{3};
\draw (1,-0.5) node{\rm (iv). $\{\{1,2\},~\{1,3\},~\{2,3\}\}$};
\draw (1,1.73)-- (0,0)--(2,0)--(1,1.73);
\end{tikzpicture}
\end{center}
In the above picture, the indexing sets under these complexes denote their maximal simplices and represent their maximal parabolic subgroups of finite type. For the simplicial representation, a given simplicial complex represents a class of Kac-Moody groups with infinite type. Those faces that are in this complex represent their parabolic subgroups of finite type, and those simplices that are not in this complex represent their parabolic subgroups of infinite type. These restrictions uniquely determine the class of Kac-Moody groups.

\begin{Prop}\label{s3ht}
For the above four classes, the homotopy types of classifying spaces of infinite Kac-Moody groups are listed as follows.
\begin{enumerate}
\item[{\rm (i)}] $BK(A)\simeq X_1\cup X_2 \cup X_3;$
\item[{\rm (ii)}] $BK(A)\simeq X_{12}\cup X_3;$
\item[{\rm (iii)}] $BK(A)\simeq X_{12}\cup X_{13};$
\item[{\rm (iv)}] $BK(A)\simeq X_{12}\cup X_{13}\cup X_{23}.$
\end{enumerate}
\end{Prop}

\subsection{An important result for finite Kac-Moody groups }
To compute the cohomology of $BK(A)$ for each class of Proposition \ref{s3ht}, we may use the Mayer-Vietoris sequence. For the case that the coefficient is a field, we could refine the Mayer-Vietoris sequence into a split short exact sequence, which will be used repeatedly in the computation of cohomology groups of $BK(A)$.
{\begin{Lem}\label{s3l1}
For the push-out diagram:
$$\CD
 X_1\cap X_2@>j_1>>X_1 \\
  @V j_2 VV @V i_1 VV  \\
  X_2 @>i_2>> X_1\cup X_2,
\endCD$$
there is a long exact sequence :
$$\cdots\stackrel{i^{*}}\to H^{*}(X_{1})\oplus  H^{*}(X_{2})\stackrel{j^{*}}\to H^{*}(X_{1}\cap X_{2})\stackrel{\delta^{*}}\to H^{*+1}(X_{1}\cup X_{2})\stackrel{i^{*+1}}\to$$
and a short exact sequence :
$$\xymatrix@C=0.5cm{
  0 \ar[r] &\Sigma {\rm coker} j^{*-1} \ar[rr] &&H^{*}(X_{1}\cup X_{2})\ar[rr] &&\ker j^{*}  \ar[r] & 0 }$$
 where the homomorphism $j: H^{*}(X_{1})\oplus  H^{*}(X_{2})\to H^{*}(X_{1}\cap X_{2})$ is given by $j(u,v)=j_1^*(u)-j_2^*(v)$.
\end{Lem}
\begin{proof}
We can immediately obtain the long exact sequence, since $(X_{1},~X_{2})$ is a Mayer-Vietoris pair. By the exactness of this sequence, we have $$\mathrm{im}\,i^{*}\cong \ker\,j^{*}, \mathrm{im}\,\delta^{*-1}\cong H^{*-1}( X_1\cap X_2)/\ker\,\delta^{*-1}\cong H^{*-1}( X_1\cap X_2)/ \mathrm{im}\,j^{*-1}\cong \mathrm{coker}\,j^{*-1}.$$
The second isomorphism $\Sigma: \mathrm{coker}\,j^{*-1} \cong \mathrm{im}\,\delta^{*-1}$ is induced by $\delta^{*-1}$.This proves the lemma.
\end{proof}}

The input of the smallest piece in the Mayer-Vietoris sequences for computing the cohomology of $BK(A)$ is the cohomology of the classifying space of finite Kac-Moody groups (compact Lie groups). The following important result for finite Kac-Moody groups is due to Borel.
\begin{Thm}[\cite{Bor53}]\label{s2t1}
Let $K$ be a compact Lie group and $T$ be its maximal torus subgroup. Let $W$ be the Weyl group of $K$. Then the inclusion $ T\hookrightarrow K$ of groups induces an injective homomorphism $\rho^{*}_{\mathbb{Q}}: H^{*}(BK;\,\mathbb{Q}) \rightarrow H^{*}(BT;\,\mathbb{Q})$ and the image of $\rho^{*}_{\mathbb{Q}}$ is $H^{*}(BT;\,\mathbb{Q})^{W}$.
\end{Thm}
\begin{Rem}\label{s2r1}
If $H^{*}(BK;\,\mathbb{Z})$ is $p$-torsion free, then the homomorphism $\rho^{*}_{\mathbb{F}_{p}}: H^{*}(BK;\,\mathbb{F}_{p}) \rightarrow H^{*}(BT;\,\mathbb{F}_{p})$ is also injective and the image of $\rho^{*}_{\mathbb{F}_{p}}$ is $H^{*}(BT;\,\mathbb{Z})^{W}\otimes\mathbb{F}_{p}\cong H^{*}(BT;\,\mathbb{F}_{p})^{W}$. See \cite[p. 400, Theorem 3.29]{MT91} for more details.
\end{Rem}

\section{Rational and mod \texorpdfstring{$p(p>2)$}{} cohomology groups \label{s3}}
In this section, we compute the rational and mod $p$ ($p>2$) cohomology groups of their classifying spaces respectively.

\subsection{The rational cohomology groups}

Recall the action of Weyl group $W(A)$ on $BT=EK(A)/T$ in subsection \ref{s2.1}, we have the following lemma.
\begin{Lem}\label{s3l2}
Let $K(A)$ be a Kac-Moody group and $T$ be its maximal torus subgroup. The map $\rho: BT \rightarrow BK(A)$ between classifying spaces is induced by the inclusion $T\hookrightarrow K(A)$. If $\rho^{*}$ is the cohomology homomorphism induced by $\rho$, then $\mathrm{im}\,\rho^{*}\subseteq H^{*}(BT; \mathbb{Z})^{W(A)}$.
\end{Lem}
\begin{proof}
For any $\varphi \in W(A)$ represented by $s \in N(T)$ and $bT \in BT$, we have
$$\rho(bT)= bK(A)= bsK(A)= \rho(bsT)= \rho(\varphi\circ bT),$$
which implies $\rho\circ\varphi= \rho$. This relation induces a cohomology homomorphism relation:
$\varphi^{*}\circ\rho^{*} =\rho^{*}$, so $\mathrm{im}\,\rho^{*}\subseteq H^{*}(BT)^{W(A)} $. This finishes the proof of the lemma.
\end{proof}

From Lemma \ref{s3l1}, we immediately obtain a cohomology commutative diagram.
\begin{Lem}\label{s3l3}
Let $K(A)$ be a Kac-Moody group, and $T$ be its maximal torus subgroup. If $J_{1},J_{2} \in \cat P(A)$, then there is an integral cohomology
commutative diagram
$$\hspace{-1.5cm}\xymatrix{
& H^{*-1}(X_{J_{1}\cap J_{2}}) \ar[d]_{\rho^{*}_{J_{1}\cap J_{2}}} \ar[r]^{\delta} & H^{*}(X_{J_{1}}\cup X_{J_{2}}) \ar[d]_{\rho^{*}_{J_{1}\cup J_{2}}} \ar[r]^{\hspace{-0.6cm}i^{*}} & H^{*}(X_{J_{1}})\oplus H^{*}(X_{J_{2}}) \ar[d]_{\rho^{*}_{J_{1}}\oplus \rho^{*}_{J_{2}}} \ar[r]^{\hspace{0.9cm}j^{*}} & H^{*}(X_{J_{1}\cap J_{2}}) \ar[d]_{\rho^{*}_{J_{1}\cap J_{2}}} & \\
& H^{*-1}(BT)^{W_{J_{1}\cap J_{2}}(A)} \ar[r]^{\delta} &H^{*}(BT)^{W_{J_{1}\cup J_{2}}(A)} \ar[r]^{\hspace{-0.8cm}i^{*}} &H^{*}(BT)^{W_{J_{1}}(A)}\oplus H^{*}(BT)^{W_{J_{2}}(A)} \ar[r]^{\hspace{0.9cm}j^{*}} & H^{*}(BT)^{W_{J_{1}\cap J_{2}}(A)}. & }$$
\end{Lem}
\begin{proof}
From Lemma \ref{s3l2}, it follows that the homomorphisms $\rho^{*}_{J_{1}}$, $\rho^{*}_{J_{2}}$, and $\rho^{*}_{J_{1}\cap J_{2}}$ are well defined. Note that
\begin{align*}
H^{*}(BT)^{W_{J_{1}\cup J_{2}}(A)} &\cong H^{*}(BT)^{W_{J_{1}}(A)}\cap H^{*}(BT)^{W_{J_{2}}(A)} \\
&\cong \{(u,~v)\in H^{*}(BT)^{W_{J_{1}}(A)}\oplus H^{*}(BT)^{W_{J_{2}}(A)} \} \cong \ker\;j^{*},
\end{align*}
which, combined with Lemma \ref{s3l1}, implies that $\rho^{*}_{J_{1}\cup J_{2}}$ is well defined. By the naturality property, the above diagram is commutative, which completes the proof.
\end{proof}

According to the classification \ref{s3ht} of rank 3 infinite Kac-Moody groups, we obtain the rational cohomology groups of their classifying spaces for each class as follows;
\begin{Thm}\label{s3t1}
For rank 3 infinite Kac-Moody groups, the rational cohomology groups of their classifying spaces are represented as follows:
\begin{enumerate}
\item[{\rm (i)}] $\Sigma(P-(P_{1}+P_{2}))\oplus \Sigma(P-(P_{12}+P_{3}))\oplus P_{123};$
\item[{\rm (ii)}] $\Sigma(P-(P_{12}+P_{3}))\oplus P_{123};$
\item[{\rm (iii)}] $\Sigma(P_{1}-(P_{12}+P_{13}))\oplus P_{123};$
\item[{\rm (iv)}] $\Sigma^{2}(P-(P_{1}+P_{2}+P_{3}))\oplus \Sigma(P_{1}\cap (P_{2}+P_{3})-(P_{12}+P_{13}))\oplus P_{123}.$
\end{enumerate}
\end{Thm}
\begin{proof}
We only prove the most complicated case (iv), and the proof of the other cases is similar. Recall from Proposition \ref{s3ht} that
$$BK(A) \simeq X_{12}\cup X_{13}\cup X_{23}.$$

By Lemma \ref{s3l1}, we obtain $$H^{*}(X_{12}\cup X_{13};\,\mathbb{Q})\cong \Sigma \mathrm{coker}\,j^{*-1}_{3}\oplus \ker\,j^{*}_{3},$$
where $$j^{*}_{3}: H^*(X_{12};\,\mathbb{Q})\oplus H^*(X_{13};\,\mathbb{Q}) \stackrel{}\to H^*(X_{1};\,\mathbb{Q}).$$
In addition, we have the continuous maps
$$\rho_{12}: BT\hookrightarrow BK_{12}(A)~,~~\rho_{13}: BT\hookrightarrow BK_{13}(A)~,~{\rm and}~~\rho_{1}: BT\hookrightarrow BK_{1}(A)$$
induced by the corresponding homomorphisms of groups. After homotopy replacement, we still use the same notations to denote the corresponding continuous maps. By Lemma \ref{s3l3}, we have the following commutative diagram of cohomology homomorphisms induced by these continuous maps
$$\xymatrix{
H^*(X_{12};\,\mathbb{Q})\ar[d]_{\rho^*_{12}}\hspace{-0.8cm}&\oplus&\hspace{-0.8cm}H^*(X_{13};\,\mathbb{Q}) \ar[d]_{\rho^*_{13}}\ar[r]^{j^{*}_{3}}& H^*(X_{1};\,\mathbb{Q})\ar[d]^{\rho^*_{1}} \\
P_{12}\hspace{-0.4cm}&\oplus&\hspace{-0.3cm}P_{13}\ar[r] & P_{1}.}$$
From Theorem \ref{s2t1} and the fact that $\{1,\,2\},\,\{1,\,3\}$ and $\{1\}$ correspond to finite parabolic subgroups, we deduce that $\rho^*_{12},~\rho^*_{13}$ and $\rho^*_{1}$ are isomorphisms. In addition, by identifying these cohomologies with the invariants of Weyl groups, we know that
$j^{*}_{3}: P_{12}\oplus P_{13}\stackrel{}\to P_{1}$ is given by $j^{*}_{3}(u,~v)= u-v$, for $u \in P_{12}, v \in P_{13}.$
And we have
$$\ker\,j^{*}_{3}=\{(u,~ v)\in P_{12}\oplus P_{13}|~ u=v\}\cong P_{12}\cap P_{13}= P_{123}~~~{\rm and}~~~\mathrm{coker}\,j^{*}_{3}= P_{1}-(P_{12}+P_{13}).$$
Thus, we obtain
$$H^{*}(X_{12}\cup X_{13};\,\mathbb{Q})\cong \Sigma(P_{1}-(P_{12}+P_{13}))\oplus P_{123}.$$

Now, we compute the cohomology groups of $X_{2}\cup X_{3}$ (which is the intersection of $X_{12}\cup X_{13}$ and $X_{23}$). Similar to the computation of $H^{*}(X_{12}\cup X_{13};\,\mathbb{Q})$ above, we get
$$H^{*}(X_{2}\cup X_{3}; \,\mathbb{Q})\cong \Sigma \mathrm{coker}\,j^{*-1}_{4}\oplus \ker\;j^{*}_{4} \cong \Sigma(P-(P_{2}+P_{3}))\oplus P_{23}.$$

With the same analysis, we obtain
$$H^{*}(BK(A);\,\mathbb{Q})\cong H^{*}((X_{12}\cup X_{13})\cup X_{23};\,\mathbb{Q}) \cong \Sigma\mathrm{coker}\,j_{5}^{*-1} \oplus \ker j_{5}^{*},$$
where $$j_{5}^{*}:H^{*}(X_{12}\cup X_{13};\,\mathbb{Q})\oplus H^{*}(X_{23};\,\mathbb{Q}) \to H^{*}(X_{2}\cup X_{3};\,\mathbb{Q}).$$
And we have the following commutative diagram
$$\xymatrix{H^{*}(X_{12}\cup X_{13};\,\mathbb{Q})\ar[d]_{\cong}\hspace{-1.2cm}&\oplus&\hspace{-0.8cm}H^*(X_{23};\,\mathbb{Q}) \ar[d]_{\rho^*_{23}}^{\cong}\ar[r]^{\hspace{-0.6cm}j^{*}_{5}}& H^{*}(X_{2}\cup X_{3};\,\mathbb{Q})\ar[d]^{\cong} \\
\Sigma(P_{1}-(P_{12}+P_{13}))\oplus P_{123}\hspace{-0.8cm}&\oplus&\hspace{-0.3cm}P_{23}\ar[r] & \Sigma(P-(P_{2}+P_{3}))\oplus P_{23}.}$$
Moreover, by the facts that $$\Sigma(P_{1}-(P_{12}+P_{13}))~~~{\rm and}~~~\Sigma(P-(P_{2}+P_{3}))$$ are the only odd degree parts and $j_{5}^{*}$ preserves degrees, we decompose $j_{5}^{*}$ into two homomorphisms $$j^{*}_{6}: \Sigma(P_{1}-(P_{12}+P_{13}))\to \Sigma(P-(P_{2}+P_{3}))~~~{\rm and}~~~j^{*}_{7}: P_{123}\oplus P_{23}\to P_{23}.$$
Then, note that the continuous map $f: (X_{2}\cup X_{3};~X_{2},~X_{3})\rightarrow (X_{12}\cup X_{13};~X_{12},~X_{13})$ induces the following commutative diagram between two Mayer-Vietoris sequences,
$$\xymatrix{
H^{*-1}( X_{1};\,\mathbb{Q})\ar[d]\ar[r]^{\delta^{*-1}}& H^{*}( X_{12}\cup X_{13};\,\mathbb{Q}) \ar[d]_{f^*} \ar[r]^{\hspace{-0.5cm}i_{1}^*}& H^{*}(X_{12};\,\mathbb{Q})\oplus H^{*}(X_{13};\,\mathbb{Q}) \ar[d] \\
H^{*-1}( X_{\emptyset};\,\mathbb{Q})\ar[r]^{\delta^{*-1}}& H^{*}( X_{2}\cup X_{3};\,\mathbb{Q})\ar[r]^{\hspace{-0.5cm}i_{2}^*}& H^{*}(X_{2};\,\mathbb{Q})\oplus H^{*}(X_{3};\,\mathbb{Q}).}$$
From the right square in the above diagram, it follows that $j^{*}_{7}(u,~v)= u-v$ for any $u \in P_{123}$, $v \in P_{23}$.
Additionally, by the left commutative square, we find that $j^{*}_{6}(\Sigma (w+P_{12}+P_{13}))=\Sigma (w+P_{2}+P_{3})$ for any $w \in P_{1}$. Then, we also decompose the corresponding kernel part and cokernel part as follows,
\begin{align*}
&\Sigma\mathrm{coker}\,j_{5}^{*-1} \oplus \ker\,j_{5}^{*}\\
&\cong \Sigma \mathrm{coker}\,j^{*-1}_{6} \oplus \Sigma \mathrm{coker}\,j^{*-1}_{7} \oplus \ker\;j^{*}_{6} \oplus \ker\,j^{*}_{7} \\
&\cong \Sigma^{2}((P-(P_{2}+P_{3}))-j^{*-1}_{6}(P_{1}-(P_{12}+P_{13})))\oplus \Sigma(P_{1}\cap (P_{2}+P_{3})-(P_{12}+P_{13}))\oplus P_{123}.
\end{align*}

Furthermore, we find that $\Sigma \mathrm{coker}~j^{*-1}_{6}$ can be reduced. From easy observation, we have
$$j^{*-1}_{6}(P_{1}-(P_{12}+P_{13})))= (P_{1}+P_{2}+P_{3})-(P_{2}+P_{3})$$ and
$$(P-(P_{2}+P_{3}))-((P_{1}+P_{2}+P_{3})-(P_{2}+P_{3})) \cong P-(P_{1}+P_{2}+P_{3}).$$
Therefore, we conclude that
$$H^{*}(BK(A);\,\mathbb{Q})\cong \Sigma^{2}(P-(P_{1}+P_{2}+P_{3}))\oplus \Sigma(P_{1}\cap (P_{2}+P_{3})-(P_{12}+P_{13}))\oplus P_{123},$$
which finishes the proof.\end{proof}

\subsection{The mod \texorpdfstring{$p(p>2)$} cohomology groups}

The computations of the rational cohomology and the mod $p$ cohomology are very similar, but the difference is that whether to consider $p$-torsion s in the input of the cohomology groups of the classifying space of those finite parabolic subgroups.  If the cohomology groups of the classifying space of all those finite parabolic subgroups are
$p$-torsion free, then the computation of the mod $p$ case is the same as that of the rational case.
That's the reason why we divide the mod $p$ cohomology groups into two cases: $p>2$ and $p=2$.
Because of low rank, there is only a few cases that we need to consider a $p$-torsion s in the input.

Let $K(A)$ be a simply connected rank 3 Kac-Moody group, then $K(A_J)$ is also simply connected for any $J\subseteq \{1,\,2,\,3\}$, see \cite{ZRW24} for details. Furthermore, according to
\cite[Theorem 5.11]{MT91}, there is a tabulation of $p$-torsion information about the integral cohomology groups of classifying spaces of simply connected compact (simple) Lie groups as follows:
\begin{center}\label{s3ta}
\begin{tabular}{|c|c|c|c|c|c|c|c|c|c|}
\hline Lie groups&$A_{l}~(l\geq 1)$&$B_{l}~(l\geq 3)$&$C_{l}~(l\geq 3)$&$D_{l}~(l\geq 4)$&$E_{6}$&$E_{7}$&$E_{8}$&$F_{4}$&$G_{2}$\\
\hline $p$-torsion& none &2 & none &2 &2,3 &2,3 &2,3,5 &2,3 &2 \\
\hline
\end{tabular}
\end{center}
Thus, for proper parabolic subgroups of rank 3 infinite Kac-Moody groups, there is just one possible case: $G_{2}$ for 2-torsion. Therefore, for the case $p=2$, we have to deal with the rank 2 parabolic subgroup $G_{2}$. For the case $p$ ($p>2$), our method to compute the rational cohomology can also be used for mod $p$ case. The key point is that the proper parabolic subgroups of rank 3 infinite Kac-Moody groups have no a $p$-torsion  for $p>2$, and the details will be revealed in the following.

\begin{Lem}\label{s3l5}
Let $K(A)$ be a rank 3 infinite Kac-Moody group, and $J$ be a proper subset of $\{1,\,2,\,3\}$ such that $A_{J}$ is a finite Cartan matrix. Then for $p>2$,
$$H^{*}(BK_{J}(A);\,\mathbb{F}_{p})\cong H^{*}(BT;\,\mathbb{F}_{p})^{W_{J}(A)}.$$
\end{Lem}
\begin{proof}
There is a short group exact sequence:
$$\xymatrix@C=0.5cm{
  1 \ar[r] &K(A_{J}) \ar[rr] &&K_{J}(A)\ar[rr] &&T^{3-|J|}  \ar[r] & 1,}$$
which induces a fibration:
$$BK(A_{J}) \rightarrow BK_{J}(A)\rightarrow BT^{3-|J|}.$$
Obviously, $H^{*}(BT^{3-|J|};\,\mathbb{F}_{p})$ is a polynomial algebra. Then from the facts that $A_{J}$ is a finite Cartan matrix and $H^{*}(BK(A_{J});\,\mathbb{Z})$ is $p$-torsion free for $p>2$, and Remark \ref{s2r1}, we deduce that $$H^{*}(BK(A_{J});\,\mathbb{F}_{p})\cong H^{*}(BT^{|J|};\,\mathbb{F}_{p})^{W(A_{J})}.$$ Observing that $W_{J}(A)\cong W(A_{J})$ is a finite group, we conclude that $H^{*}(BT;\,\mathbb{F}_{p})^{W_{J}(A)}$ is also a polynomial algebra by Hilbert's finiteness theorem \cite{Hil93}. Besides, both generators of these two polynomial algebras are in even degree. Thus, the Serre spectral sequence of the fibration collapses at $E^{2}$-term, in other words, $E^{2} \cong E^{\infty}$, which implies that
\begin{align*}
H^{*}(BK_{J}(A);\,\mathbb{F}_{p})&\cong H^{*}(BK(A_{J});\,\mathbb{F}_{p})\otimes H^{*}(BT^{3-|J|};\,\mathbb{F}_{p})\\
&\cong H^{*}(BT^{|J|};\,\mathbb{F}_{p})^{W(A_{J})}\otimes H^{*}(BT^{3-|J|};\,\mathbb{F}_{p}) \cong H^{*}(BT;\,\mathbb{F}_{p})^{W_{J}(A)}.
\end{align*}
This finishes the proof of Lemma \ref{s3l5}.
\end{proof}

\begin{Rem}\label{s3r1}
From the proof of Lemma \ref{s3l5}, it follows that, if $H^{*}(BK(A_{J});\,\mathbb{Z})$ has no $2$-torsion with the same notations as in Lemma \ref{s3l5}, then $H^{*}(BK_{J}(A);\,\mathbb{F}_{2})\cong H^{*}(BT;\,\mathbb{F}_{2})^{W_{J}(A)}$.
\end{Rem}
The following theorem for the mod $p$ case is also similar to Theorem \ref{s3t1}.
\begin{Thm}\label{s3t2}
The mod $p$ ($p>2$) cohomology groups of classifying spaces of rank 3 infinite Kac-Moody groups have the same form as the rational case.
\end{Thm}
\begin{proof}
As we can see from Lemma \ref{s3l5}, for $p>2$, the method to compute the rational cohomology is also true for the mod $p$ case. Therefore, we omit the details.
\end{proof}

\section{Mod 2 cohomology groups \label{s4}}

\subsection{The mod 2
cohomology groups}
For a rank 3 infinite Kac-Moody group $K(A)$ with the maximal torus subgroup $T$, the cohomologies of the classifying spaces of its finite parabolic subgroups are the input for the computation of $ H^{*}(BK(A))$. In particular, for rational and mod $p$ ($p>2$) cases, this input can be computed by Theorem \ref{s2t1} and Lemma \ref{s3l5}. Since $G_2$ may appear as a parabolic subgroup of $K(A)$, and we have no similar results to Theorem \ref{s2t1} and Lemma \ref{s3l5} to compute the mod 2 cohomology groups of its classifying space, this raises a problem for the mod 2 case. Without loss of generality, we suppose the index of this parabolic subgroup is \{1,\,2\} and denote $K_{12}(A)$ the parabolic subgroup in $K(A)$. The isomorphism type of $K_{12}(A)$ has been determined in \cite[Example 3.2]{ZRW24}. Furthermore, $$K_{12}(A)\cong K(A_{12})\times T^{1}\cong G_{2}\times T^{1}.$$
Then, applying the K$\ddot{\rm u}$nneth formula, we conclude that
\begin{equation}\label{s4f1}
H^*(B(G_{2}\times T^{1});\,\mathbb{F}_{2})\cong H^*(BG_{2};\,\mathbb{F}_{2})\otimes H^*(BT^{1};\,\mathbb{F}_{2}).
\end{equation}
In fact, the key for our computation is the analysis of the homomorphism $$\rho^*_{12} : H^{*}(B(G_{2}\times T^{1});\,\mathbb{F}_{2})\rightarrow H^{*}(BT;\,\mathbb{F}_{2})$$ induced by the inclusion $\rho_{12}: BT=B(T^2\times T^{1})\hookrightarrow B(G_{2}\times T^{1})$, and $\rho^*_{12}$ could be reduced to the homomorphism $p^*_{\mathbb{F}_{2}} : H^{*}(BG_{2};\,\mathbb{F}_{2})\rightarrow H^{*}(BT^2;\,\mathbb{F}_{2})$ induced by the inclusion $p: T^2 \hookrightarrow G_{2}$. Moreover, we may choose a subroot system of $G_{2}$ that contains all long roots which give rise to a subgroup $\textit{SU(3)}$. Thus, $p^*_{\mathbb{F}_{2}}$ can be viewed as the composite homomorphism
\begin{align}\label{s4fo1}
H^{*}(BG_{2};\,\mathbb{F}_{2})\stackrel{p^{*}_1}\to H^{*}(\textit{BSU(3)};\,\mathbb{F}_{2}) \stackrel{p^{*}_2}\to H^{*}(BT^2;\,\mathbb{F}_{2})
\end{align}
induced by the inclusions of groups. Since $H^{*}(B\textit{SU(3)};\,\mathbb{Z})$ is 2-torsion free, combined with Remark \ref{s2r1}, we deduce that $p^{*}_2$ is injective and $$H^{*}(\textit{BSU(3)};\,\mathbb{F}_{2})\cong H^{*}(BT^2;\,\mathbb{F}_{2})^{\textit{W(SU(3))}},$$
namely, \begin{align}\label{s4fo2}
\ker\;p^{*}_2 =0~{\rm and}~\mathrm{im}\,p^{*}_2\cong H^{*}(BT^2;\,\mathbb{F}_{2})^{\textit{W(SU(3))}}.
\end{align}
Furthermore, by the fact that Cartan matrices of $G_{2}$ and $\textit{SU(3)}$ are the same after mod 2, we obtain
\begin{align}\label{s4fo3}
H^{*}(BT^2;\,\mathbb{F}_{2})^{\textit{W(SU(3))}}\cong H^{*}(BT^2;\,\mathbb{F}_{2})^{W(G_{2})}.
\end{align}
Now, we deal with $\ker p^{*}_{\mathbb{F}_{2}}$.
By \cite[Chapter 3. Theorem 3.17(2) and Chapter 7. Corollary 6.3]{MT91}, we know that
\begin{align*}
H^{*}(\textit{BSU(3)};\,\mathbb{F}_{2})= \mathbb{F}_{2}[c_{2},~c_{3}],~\deg c_{2}=4,~\deg c_{3}=6,~\text{and}
\end{align*}
\begin{align}
H^{*}(BG_{2};\,\mathbb{F}_{2})=\mathbb{F}_{2}[y_{4},~y_{6},~y_{7}],~\deg y_{4}= 4,~y_{6}= {\rm Sq}^{2}y_{4},~y_{7}= {\rm Sq}^{3}y_{4}= {\rm Sq}^{1}y_{6}\label{s4f2}.
\end{align}
Moreover, in \ref{s4fo1}, we find that $p^{*}_1$ is surjective by our construction on the root system. Additionally, from the fact that it preserves degrees, we have $p^{*}_{1}(y_{4})= c_{2}$, $p^{*}_{1}(y_{6})= c_{3} $, and $p^{*}_{1}(y_{7})= 0$. This, combined with \ref{s4fo2} and \ref{s4fo3}, implies that $\mathrm{im} \,p^{*}_{\mathbb{F}_{2}}$ is the invariant $H^{*}(BT^2;\,\mathbb{F}_{2})^{W(G_{2})}$ and $y_{7}\in \ker\;p^{*}_{\mathbb{F}_{2}}$. Therefore, we conclude that $\mathrm{im}\,p^{*}_{\mathbb{F}_{2}}\cong H^{*}(BT^2;\,\mathbb{F}_{2})^{W(G_{2})}$, and $\ker\,p^{*}_{\mathbb{F}_{2}}$ is the ideal $(y_{7})$ in $H^{*}(BG_{2};\,\mathbb{F}_{2})$.

We have understood the behavior of $p^*_{\mathbb{F}_{2}}$ for $G_{2}$. Then, we have to understand the behavior of $\rho^*_{12}$ for $K_{12}(A)$. Previously, we have constructed a $\textit{SU(3)}$ as its subgroup by the root system. Similarly, for $K_{12}(A)$, we construct a subgroup $K_{12}'(A)$ that corresponds to $\textit{SU(3)}$ with the Weyl group $W'_{12}(A)$. From the fact that $H^{*}(\textit{BSU(3)};\,\mathbb{Z})$ is 2-torsion free, combined with Remark \ref{s3r1}, it follows that
$$H^{*}(BK_{12}'(A);\,\mathbb{F}_{2})\cong H^{*}(BT;\,\mathbb{F}_{2})^{W'_{12}(A)} .$$
Moreover, by easy observation, we have
\begin{equation}\label{s4f3}
H^{*}(BT;\,\mathbb{F}_{2})^{W'_{12}(A)} \cong H^{*}(BT^2;\,\mathbb{F}_{2})^{\textit{W(SU(3))}}\otimes H^*(BT^{1};\,\mathbb{F}_{2}).
\end{equation}
Since $\rho_{12}$ factors through $K_{12}'(A)$, $\rho_{12}^*$ can be viewed as the composite homomorphism
$$\hspace{-1cm}\xymatrix{
H^{*}(BK_{12}(A);\,\mathbb{Z}_{2}) \ar[d]_{\cong} \ar[r]^{p^{*}_1\otimes {\rm id}} & H^{*}(BK_{12}'(A);\,\mathbb{F}_{2}) \ar[d]_{\cong}\ar[r]^{p^{*}_2\otimes {\rm id}} & H^{*}(BT;\,\mathbb{F}_{2})\ar[d]_{\cong}\\
H^*(BG_{2};\,\mathbb{F}_{2})\otimes H^*(BT^{1};\,\mathbb{F}_{2})\ar[d]_{\cong} & H^{*}(BT^2;\,\mathbb{F}_{2})^{\textit{W(SU(3))}}\otimes H^*(BT^{1};\,\mathbb{F}_{2})\ar[d]_{\cong} &H^{*}(BT^2;\,\mathbb{F}_{2})\otimes H^{*}(BT^1;\,\mathbb{F}_{2})\ar[d]_{\cong}\\
\mathbb{F}_{2}[y_{4},~y_{6},~y_{7},~\omega_{3}] &\mathbb{F}_{2}[c_{2},~c_{3},~\omega_{3}] & \mathbb{F}_{2}[\omega_{1},~\omega_{2},~\omega_{3}].}$$

We summarize the above argument into the following lemma.
\begin{Lem}\label{s4l1}
Let $K(A)$ be a rank 3 infinite Kac-Moody group. If $G_2$ occurs as a parabolic subgroup of $K(A)$ indexed by $\{1,~2\}$, then $\ker\,\rho_{12}^*$ is the ideal $(y_{7})$ in $\mathbb{F}_{2}[y_{4},~y_{6},~y_{7},~\omega_{3}]$, and $\mathrm{im}\,\rho_{12}^*\cong H^{*}(BT;\,\mathbb{F}_{2})^{W_{12}(A)}$.
\end{Lem}
Due to symmetry of indices, if $G_2$ occurs as a parabolic subgroup of $K(A)$ indexed by $\{1,~3\}$ or $\{2,~3\}$, then $\ker\,\rho_{13}^*$ or $\ker\,\rho_{23}^*$ is the ideal $(y'_{7})$ in $\mathbb{F}_{2}[y'_{4},~y'_{6},~y'_{7},~\omega_{2}]$ or $(y''_{7})$ in $\mathbb{F}_{2}[y''_{4},~y''_{6},~y''_{7},~\omega_{1}]$.

Based on how $G_{2}$ appears as a rank 2 finite parabolic subgroup in $P(A)$, we give a refined classification of rank 3 infinite Kac-Moody groups.
\begin{Prop}\label{M2C}
Up to symmetry, there are ten classes as follows:
\begin{enumerate}
\item[{\rm (i)}] $\{1\},~\{2\},~\{3\};$
\item[{\rm (ii)}] $\{1, 2\},~\{3\},~K(A_{12})\ncong G_{2};$
\item[{\rm (iii)}]$\{1, 2\},~\{3\},~K(A_{12})\cong G_{2};$
\item[{\rm (iv)}] $\{1, 2\},~\{1, 3\},~K(A_{12})\ncong G_{2},~K(A_{13})\ncong G_{2};$
\item[{\rm (v)}] $\{1, 2\},~\{1, 3\},~K(A_{12})\cong G_{2},~K(A_{13})\ncong G_{2};$
\item[{\rm (vi)}] $\{1, 2\},~\{1, 3\},~K(A_{12})\cong K(A_{13})\cong G_{2};$
\item[{\rm (vii)}] $\{1, 2\},~\{1, 3\},~\{2, 3\},~K(A_{12})\ncong G_{2},~K(A_{13})\ncong G_{2},~K(A_{23})\ncong G_{2};$
\item[{\rm (viii)}]$\{1, 2\},~\{1, 3\},~\{2, 3\},~K(A_{12})\cong G_{2},~K(A_{13})\ncong G_{2},~K(A_{23})\ncong G_{2};$
\item[{\rm (ix)}]$\{1, 2\},~\{1, 3\},~\{2, 3\},~K(A_{12})\cong K(A_{13})\cong G_{2},~K(A_{23})\ncong G_{2};$
\item[{\rm (x)}] $\{1, 2\},~\{1, 3\},~\{2, 3\},~K(A_{12})\cong K(A_{13})\cong K(A_{23})\cong G_{2}.$
\end{enumerate}
\end{Prop}
Now, we compute the cohomology groups of rank 3 infinite Kac-Moody groups for the mod 2 case. The method to compute the mod 2 cohomology groups is similar to the rational case whenever $G_{2}$ appears in $P(A)$, and the cohomology groups are always split into two parts by the Mayer-Vietoris sequence: the kernel part and cokernel part. If $G_{2}$ appears, then the cohomology groups will have an extra part of the kernel.
\begin{Thm}\label{s4t1}
Denote $P_{J}$ the invariant $H^{*}(BT;\mathbb{F}_{2})^{W_{J}(A)}$ and $I_{3},~I_{2},$ and $I_{1}$ the ideals $(y_{7})$ in $\mathbb{F}_{2}[y_{4},~y_{6},~y_{7},~\omega_3]$, $(y'_{7})$ in $\mathbb{F}_{2}[y'_{4},~y'_{6},~y'_{7},~\omega_2]$, and $(y''_{7})$ in $\mathbb{F}_{2}[y''_{4},~y''_{6},~y''_{7},~\omega_1]$, respectively, while the other notations are the same as in Theorem \ref{s3t1}. By the above classification, the mod 2 cohomology groups of classifying spaces of rank 3 infinite Kac-Moody groups can be represented as follows:
\begin{enumerate}
\item[{\rm (i)}] $\Sigma(P-(P_{1}+P_{2}))\oplus \Sigma(P-(P_{12}+P_{3}))\oplus P_{123};$
\item[{\rm (ii)}] $\Sigma(P-(P_{12}+P_{3}))\oplus P_{123};$
\item[{\rm (iii)}]$\Sigma(P-(P_{12}+P_{3}))\oplus P_{123}\oplus I_{3};$
\item[{\rm (iv)}] $\Sigma(P_{1}-(P_{12}+P_{13}))\oplus P_{123};$
\item[{\rm (v)}] $\Sigma(P_{1}-(P_{12}+P_{13}))\oplus P_{123}\oplus I_{3};$
\item[{\rm (vi)}] $\Sigma(P_{1}-(P_{12}+P_{13}))\oplus P_{123}\oplus I_{3}\oplus I_{2};$
\item[{\rm (vii)}] $\Sigma^{2}(P-(P_{1}+P_{2}+P_{3}))\oplus \Sigma(P_{1}\cap (P_{2}+P_{3})-(P_{12}+P_{13}))\oplus P_{123};$
\item[{\rm (viii)}]$\Sigma^{2}(P-(P_{1}+P_{2}+P_{3}))\oplus \Sigma(P_{1}\cap (P_{2}+P_{3})-(P_{12}+P_{13}))\oplus P_{123}\oplus I_{3};$
\item[{\rm (ix)}]$\Sigma^{2}(P-(P_{1}+P_{2}+P_{3}))\oplus \Sigma(P_{1}\cap (P_{2}+P_{3})-(P_{12}+P_{13}))\oplus P_{123}\oplus I_{3}\oplus I_{2};$
\item[{\rm (x)}]$\Sigma^{2}(P-(P_{1}+P_{2}+P_{3}))\oplus \Sigma(P_{1}\cap (P_{2}+P_{3})-(P_{12}+P_{13}))\oplus P_{123}\oplus I_{3}\oplus I_{2}\oplus I_{1}.$
\end{enumerate}
\end{Thm}
\begin{proof}
The proof of this theorem is similar to Theorem \ref{s3t1}. We only show the computation of case (x) with details. The same notations are used as in the proof of Theorem \ref{s3t1}. By homotopy replacement, we first have
$$BK(A) \simeq X_{12}\cup X_{13}\cup X_{23}.$$
Without affecting the actual cohomology of $BK(A)$, we prefer to first paste $X_{12}$ and $X_{13}$. By Lemma \ref{s3l1}, we get $$H^{*}(X_{12}\cup X_{13};\,\mathbb{F}_2)\cong \Sigma \mathrm{coker}~j^{*-1}_{3}\oplus \ker\,j^{*}_{3},$$
where $$j^{*}_{3}: H^*(X_{12};\,\mathbb{F}_2)\oplus H^*(X_{13};\,\mathbb{F}_2) \stackrel{}\to H^*(X_{1};\,\mathbb{F}_{2}).$$
In addition, by Lemma \ref{s3l3}, we have the following commutative diagram
$$\xymatrix{
H^*(X_{12};\,\mathbb{F}_{2})\ar[d]_{\rho^*_{12}}\hspace{-0.7cm}&\oplus&\hspace{-0.8cm}H^*(X_{13};\,\mathbb{F}_{2}) \ar[d]_{\rho^*_{13}}\ar[r]^{j^{*}_{3}}& H^*(X_{1};\,\mathbb{F}_{2})\ar[d]^{\rho^*_{1}} \\
P_{12}\hspace{-0.4cm}&\oplus &\hspace{-0.3cm}P_{13}\ar[r] & P_{1}.}$$
Moreover, from Remark \ref{s3r1} and the fact that $\{1\}$ represents a finite parabolic subgroup and its integral cohomology has no 2-torsion, we deduce that $\rho^*_{1}$ is an isomorphism. Note that both $\{1,~2\}$ and $\{1,~3\}$ represent a $G_2$, and combined with \ref{s4f1}, we have
\begin{equation}\label{s4f4}
H^*(X_{12};\,\mathbb{F}_{2})\cong \mathbb{F}_{2}[y_{4},~y_{6},~y_{7},~\omega_3]~,~ H^*(X_{13};\,\mathbb{F}_{2})\cong \mathbb{F}_{2}[y'_{4},~y'_{6},~y'_{7},~\omega_2].
\end{equation}
Furthermore, by Lemma \ref{s4l1}, we further deduce that $\rho^*_{12}$ and $\rho^*_{13}$ are surjective with kernels $(y_{7})$ and $(y'_{7})$, respectively.
Therefore, \begin{align*}
\ker\,j^{*}_{3}=\{(u,~v)\in P_{12}\oplus P_{13}|~ u=v\}\oplus(y_{7})\oplus(y'_{7})\cong& P_{12}\cap P_{13}\oplus (y_{7})\oplus(y'_{7})\\
\cong & P_{123}\oplus (y_{7})\oplus(y'_{7})
\end{align*}
and $\mathrm{coker}~j^{*}_{3}= P_{1}-(P_{12}+P_{13}).$

Thus, we obtain that
$$H^{*}(X_{12}\cup X_{13};\,\mathbb{F}_{2})\cong \Sigma(P_{1}-(P_{12}+P_{13}))\oplus P_{123}\oplus(y_{7})\oplus(y'_{7}).$$

Now, we shall compute the cohomology groups of $X_{2}\cup X_{3}$, which is the intersection between $X_{12}\cup X_{13}$ and $X_{23}$. We obtain $H^{*}(X_{2}\cup X_{3};\,\mathbb{F}_{2})$ in the same way as above,
$$H^{*}(X_{2}\cup X_{3}; \,\mathbb{F}_{2})\cong \Sigma \mathrm{coker}\,j^{*-1}_{4}\oplus \ker\,j^{*}_{4} \cong \Sigma(P-(P_{2}+P_{3}))\oplus P_{23},$$
where $j^{*}_{4}: P_{2}\oplus P_{3}\stackrel{}\to P$ and $j^{*}_{4}(u,~v)= u-v$ for any $u \in P_{2}$, $v \in P_{3}$.

With the same analysis, we obtain
$$H^{*}(BK(A);\,\mathbb{F}_{2})\cong H^{*}((X_{12}\cup X_{13})\cup X_{23};\,\mathbb{F}_{2}) \cong \Sigma\mathrm{coker}~j_{5}^{*-1} \oplus \ker\,j_{5}^{*},$$
where $$j_{5}^{*}: H^{*}(X_{12}\cup X_{13};\,\mathbb{F}_{2})\oplus H^{*}(X_{23};\,\mathbb{F}_{2})\to H^{*}(X_{2}\cup X_{3};\,\mathbb{F}_{2}).$$
In addition, we have the following commutative diagram
$$\xymatrix{
H^{*}(X_{12}\cup X_{13};\,\mathbb{F}_{2})\ar[d]_{\cong} \hspace{-1.6cm}& \oplus H^*(X_{23};\,\mathbb{F}_{2}) \ar[d]_{\rho^*_{23}}\ar[r]^{j^{*}_{5}}& H^{*}(X_{2}\cup X_{3};\,\mathbb{F}_{2})\ar[d]^{\cong} \\
\Sigma(P_{1}-(P_{12}+P_{13}))\oplus P_{123}\oplus(y_{7})\oplus(y'_{7}) \hspace{-1.4cm}& \oplus P_{23}\ar[r] & \Sigma(P-(P_{2}+P_{3}))\oplus P_{23}.}$$
Then, by Lemma \ref{s4l1}, we conclude that $\rho^*_{13}$ is surjective and
$$\ker\,\rho^*_{23}\cong (y''_{7})\in \mathbb{F}_{2}[y''_{4},~y''_{6},~y''_{7},~\omega_1] \cong H^*(X_{23};\,\mathbb{F}_{2}).$$
Therefore, $(y''_{7}) \subseteq \ker\,j_{5}^{*}$. Since the cohomology group of $BK(A)$ is independent of the pasting order, choosing the other two pasting orders for $X_{12}\cup X_{13}\cup X_{23}$, we arrive at the fact that $(y_{7}),~(y'_{7}) \subseteq \ker\,j_{5}^{*}$ similarly. Moreover, by the fact that $\Sigma(P_{1}-(P_{12}+P_{13}))$ and $\Sigma(P-(P_{2}+P_{3}))$ are the odd degree parts and $j_{5}^{*}$ preserves degrees, we can decompose $j_{5}^{*}$ into two homomorphisms $$j^{*}_{6}: \Sigma(P_{1}-(P_{12}+P_{13}))\oplus(y_{7})\oplus(y'_{7})\to \Sigma(P-(P_{2}+P_{3}))~~~{\rm and}~~~j^{*}_{7}: P_{123}\oplus P_{23}\to P_{23}.$$
Note that the continuous map $f: (X_{2}\cup X_{3};~X_{2},~X_{3})\rightarrow (X_{12}\cup X_{13};~X_{12},~X_{13})$ induces the following commutative diagram between two Mayer-Vietoris sequences,
$$\xymatrix{
H^{*-1}( X_{1};\,\mathbb{F}_{2}) \ar[d]_{f_1^*} \ar[r]^{\delta^{*-1}}&H^{*}( X_{12}\cup X_{13};\,\mathbb{F}_{2}) \ar[d]_{f^*} \ar[r]^{\hspace{-0.6cm}i_{1}^*} &H^{*}(X_{12};\,\mathbb{F}_{2})\oplus H^{*}(X_{13};\,\mathbb{F}_{2}) \ar[d]_{f_2^*} \\
H^{*-1}( X_{\emptyset};\,\mathbb{F}_{2}) \ar[r]^{\delta^{*-1}} & H^{*}( X_{2}\cup X_{3};\,\mathbb{F}_{2})\ar[r]^{\hspace{-0.6cm}i_{2}^*} & H^{*}(X_{2};\,\mathbb{F}_{2})\oplus H^{*}(X_{3};\,\mathbb{F}_{2}),}$$
which shows the behavior of $j^{*}_{6}$ and $j^{*}_{7}$. Observe that $$H^{*}(X_{2};\,\mathbb{F}_{2})\cong P_2,~ H^{*}(X_{3};\,\mathbb{F}_{2})\cong P_3,~{\rm and}~\ker\,i_{2}^*=\Sigma(P-(P_{2}+P_{3})),$$
In addition, from the right square in the above diagram, it follows that $j^{*}_{7}(u,~v)= u-v$ for any $u \in P_{123}$, $v \in P_{23}$.
Additionally, by the left commutative square, we find that $j^{*}_{6}(\Sigma (w+P_{12}+P_{13}))=\Sigma (w+P_{2}+P_{3})$ for any $w \in P_{1}$.
Thus, we can decompose the corresponding kernel part and cokernel part as follows,
\begin{align*}
&\Sigma\mathrm{coker}~j_{5}^{*-1} \oplus \ker\,j_{5}^{*}\\
&\cong \Sigma \mathrm{coker}~j^{*-1}_{6} \oplus \Sigma \mathrm{coker}~j^{*-1}_{7} \oplus \ker\,j^{*}_{6} \oplus \ker\,j^{*}_{7}\oplus(y''_{7}).
\end{align*}

By easy computation, we have
\begin{align*}
&\Sigma \mathrm{coker}~j^{*-1}_{6}\cong \Sigma^{2}((P-(P_{2}+P_{3}))-j^{*-1}_{6}(P_{1}-(P_{12}+P_{13}))),~\Sigma \mathrm{coker}~j^{*-1}_{7}\cong 0,\\
&\ker\,j^{*}_{6}\cong\Sigma(P_{1}\cap (P_{2}+P_{3})-(P_{12}+P_{13}))\oplus (y_{7})\oplus(y'_{7}),~ \ker\,j^{*}_{7}\cong P_{123}.
\end{align*}
With the same method as in Theorem \ref{s3t1}, we find that $\Sigma \mathrm{coker}~j^{*-1}_{6}$ can be reduced to $\Sigma (P-(P_{1}+P_{2}+P_{3}))$.
Therefore, we conclude that
\begin{align*}
H^{*}(BK(A);\,\mathbb{F}_{2})\cong ~& \Sigma^{2}(P-(P_{1}+P_{2}+P_{3}))\oplus \Sigma(P_{1}\cap (P_{2}+P_{3})-(P_{12}+P_{13}))\oplus P_{123}\\
&\oplus (y_{7})\oplus(y'_{7})\oplus(y''_{7}),
\end{align*}
which finishes the proof.\end{proof}

\subsection{Steenrod operations}

Regarding our calculations in Section \ref{s3} and this Section, there are two parts of the cohomology of $H^{*}(BK(A);\,\mathbb{F}_{p})$: the kernel part and cokernel part. Recall that $\delta$ is the coboundary map of Mayer-Vietoris sequence, the cokernel part of $H^{*}(BK(A);\,\mathbb{F}_{p})$ is isomorphic to $\mathrm{im}\,\delta^{*-1}$, and with little abuse of notation we use $\Sigma$ to denote this isomorphism, $\Sigma :\mathrm{coker}~j^{*-1} \cong \mathrm{im}\,\delta^{*-1}$. The rationale here is that $\Sigma$ commutes with Steenrod operations just as the conventional suspension does with Steenrod operations. Moreover, the cohomology is represented by the invariants of Weyl groups, which is the image of cohomology homomorphism $\rho^{*}_{\mathbb{F}_{p}}: H^{*}(BK(A);\,\mathbb{F}_{p})\rightarrow H^{*}(BT;\,\mathbb{F}_{p})$ induced by the continuous map $\rho$; thus, the Steenrod operation commutes with $ \rho^{*}_{\mathbb{F}_{p}}$. Therefore, the behavior of the Steenrod operation on the cohomology of $BT$ is sufficient to describe the Steenrod operation on the cohomology of $BK(A)$.

\section{The computation of Poincar\'e series of an example for different cohomology coefficients \label{s5}}

While we do not obtain the cohomology groups of the classifying spaces of rank 3 infinite Kac-Moody groups as a standard form of abelian groups, but our result is not far from that form. To show this, we compute the  Poincar\'e series of the rational, mod 3, and mod 2 cohomology groups of the classifying space of certain Kac-Moody group in Example \ref{s2e1}, the differences among which reveal the 3-torsion and 2-torsion information of the integral cohomology groups.

The following theorem is indispensable for computating the invariant of infinite Weyl groups.
\begin{Thm}(\cite[Main Theorem]{ZJ14})\label{s5t1}
Let $A$ be a Cartan matrix of indefinite type, and $\psi~(\deg \psi=4)$ be an invariant bilinear form called the Killing form. Then
$$H^{*}(BT;\,\mathbb{Q})^{W(A)}\cong \left\{\begin{array}{ll}
\mathbb{Q}\hspace{1cm}  \text{if $A$ is nonsymmetrizable},\\
\mathbb{Q}[\psi]\hspace{0.5cm}  \text{if $A$ is symmetrizable}.
\end{array}\right.
$$
\end{Thm}

\begin{Ex}
Recall the Cartan matrix $$A=\left(
\begin{array}{ccc}
2 & -1 & -3\\
-3 & 2 & -1\\
-2 & -4 & 2
\end{array}
\right)$$ in Example \ref{s2e1}. Its simplicial representation is
\begin{center}
\bigskip
\begin{tikzpicture}
\filldraw[black] (0,0) circle (2pt) node[anchor=east]{1}
(2,0) circle (2pt) node[anchor=west]{2~.}
(1,1.73) circle (2pt)node[anchor=south]{3};
\draw (1,-0.5) node{\rm (ii). $\{\{1,2\},~\{3\}\}$};
\draw (0,0)--(2,0);
\end{tikzpicture} \hspace{0.4cm}
\end{center}
\bigskip
By Theorems \ref{s3t1}and \ref{s3t2}, we know that the rational and mod $3$ cohomology groups of its corresponding classifying spaces will be the following form
$$\Sigma(P-(P_{12}+P_{3}))\oplus P_{123}$$
with different coefficients $\,\mathbb{Q}$ and $\mathbb{F}_{3}$, respectively. Since $K(A_{12})\cong G_{2}$, $K(A)$ belongs to the refined class
$${\rm (iii)}.~\{\{1,2\},~\{3\},~K(A_{12})\cong G_{2}\}.$$
By Theorems \ref{s4t1}, the mod $2$ cohomology group will be the following form $$\Sigma(P-(P_{12}+P_{3}))\oplus P_{123}\oplus I$$ with coefficient $\mathbb{F}_{2}$.

We next compute the Poincar\'e series of the cohomology group with rational, mod 3, and mod 2 coefficients.

$\mathbf{Rational}~~~\mathbf{case:}$
Let $H^{*}(BT;\,\mathbb{Q})=P \cong \mathbb{Q}[\omega_{1},~\omega_{2},~\omega_{3}]$. We will adhere to the same notations as in section \ref{s4}. $H^{*}(BG_{2};\,\mathbb{Q})\cong \mathbb{Q}[y_{6},~y_{22}].$ Then Theorem \ref{s2t1} implies that $$P_{3}\cong \mathbb{Q}[\omega_{1},~\omega_{2},~\omega_{3}(\omega_{3}-\alpha_{3})]~~{\rm and}~~ P_{12}\cong \mathbb{Q}[\omega_{3},~y_{6},~y_{22}],$$
where $\alpha_{3}=-3\omega_1-\omega_2+2\omega_3$.
Additionally, from Theorem \ref{s5t1} and the fact that $A$ is nonsymmetrizable, it follows that $P_{123}\cong \mathbb{Q}$. Thus, the Poincar\'e series for the rational cohomology is

$\mathbf{Mod}$ $\mathbf{3}$ $\mathbf{case:}$
Let $H^{*}(BT;\,\mathbb{F}_{3})=P\cong \mathbb{F}_{3}[\omega_{1},~\omega_{2},~\omega_{3}]$. A direct computation of the invariants of Weyl
groups shows that $P_{3}\cong\mathbb{F}_{3}[\omega_{1},~\omega_{2},~\omega_{3}(\omega_{3}-\alpha_{3})]$ and $$P_{12}\cong\mathbb{F}_{3}[\omega_{3},~\omega_{1}(\omega_{1}+\omega_{3}),~\omega_{2}(\omega_{1}+2\omega_{2}+\omega_{3})(\omega_{1}+\omega_{2})(\omega_{1}
+\omega_{2}+2\omega_{3})(2\omega_{1}+\omega_{2}+\omega_{3})(\omega_{2}+\omega_{3})].$$
Recall the action of $W(A)$ on $H^{*}(BT;\,\mathbb{F}_{3})$ in Subsection \ref{s2.1}, which induces a representation of $W(A)$ in the general linear group $GL(3,\,\mathbb{F}_{3})$. Since the image of this representation is surjective, we have $$H^{*}(BT;\,\mathbb{Q})^{W(A)} \cong H^{*}(BT;\,\mathbb{Q})^{GL(3,\,\mathbb{F}_{3})} \cong \mathbb{F}_{3}[d_{36},~d_{48},~d_{52}],$$
where $d_{36},~d_{48},~d_{52}$ are so-called Dickson invariants in \cite{Dic11}.
The Poincar\'e series for the mod $3$ cohomology is
\begin{equation*}
\begin{split}
&H(~\Sigma(P-(P_{12}+P_{3}))\oplus P_{123}~)\\
=& tH(~P-(P_{12}+P_{3})~)+H(~P_{123}~)\\
=&t(H(~P~)-(H(~P_{12}~)+H(~P_{3})-H(~P_{123}~)~))+H(~P_{123}~)\\
=&t(\frac{1}{(1-t^{2})^{3}}-(\frac{1}{(1-t^{2})(1-t^{4})(1-t^{12})}+\frac{1}{(1-t^{2})^{2}(1-t^{4})}-\frac{1}{(1-t^{36})(1-t^{48})(1-t^{52})}))\\
&+\frac{1}{(1-t^{36})(1-t^{48})(1-t^{52})}\\
=&\frac{t^{3}}{(1-t^{2})^{2}(1-t^{4})}-\frac{t}{(1-t^{2})(1-t^{4})(1-t^{12})}+\frac{1+t}{(1-t^{36})(1-t^{48})(1-t^{52})}\\
=&\frac{g(t)}{(1-t^{36})(1-t^{48})(1-t^{52})},
\end{split}
\end{equation*}
where $g(t)$ is $1+2t^{7}+3t^{9}+6t^{11}+7t^{13}+11t^{15}+13t^{17}+18t^{19}+21t^{21}+27t^{23}+30t^{25}+37t^{27}+41t^{29}+49t^{31}+54t^{33}+63t^{35}+69t^{37}+78t^{39}+
84t^{41}+93t^{43}+99t^{45}+108t^{47}+115t^{49}+123t^{51}+130t^{53}+136t^{55}+141t^{57}+145t^{59}+149t^{61}+151t^{63}+154t^{65}+154t^{67}+155t^{69}+153t^{71}
+153t^{73}+149t^{75}+148t^{77}+142t^{79}+139t^{81}+131t^{83}+126t^{85}+117t^{87}+111t^{89}+102t^{91}+96t^{93}+87t^{95}+81t^{97}+72t^{99}+65t^{101}+57t^{103}
+51t^{105}+44t^{107}+39t^{109}+33t^{111}+28t^{113}+23t^{115}+19t^{117}+15t^{119}+12t^{121}+9t^{123}+6t^{125}+4t^{127}+2t^{129}+t^{131}$.

$\mathbf{Mod}$ $\mathbf{2}$ $\mathbf{case:}$
Let $H^{*}(BT;\,\mathbb{F}_{2})=P\cong \mathbb{F}_{2}[\omega_{1},~\omega_{2},~\omega_{3}]$. A direct computation of the invariants of Weyl groups shows that $$P_{3}\cong \mathbb{F}_{2}[\omega_{1},~\omega_{2},~\omega_{3}(\omega_{3}-\alpha_{3})],~
P_{12}\cong \mathbb{F}_{2}[\omega_{3},~\omega_{1}^{2}+\omega_{2}^{2}+\omega_{1}\omega_{2},~\omega_{1}\omega_{2}(\omega_{1}+\omega_{2})]$$
and $$P_{123}\cong \mathbb{F}_{2}[\omega_{1}^{2}+\omega_{2}^{2}+\omega_{1}\omega_{2},~\omega_{1}\omega_{2}(\omega_{1}+\omega_{2}),~\omega_{3}(\omega_{1}+\omega_{2}
+\omega_{3})(\omega_{2}+\omega_{3})(\omega_{1}+\omega_{3})].$$
The Poincar\'e series for the mod $2$ cohomology is
\begin{equation*}
\begin{split}
&H(~\Sigma(P-(P_{12}+P_{3}))\oplus P_{123}\oplus I~)\\
=& tH(~P-(P_{12}+P_{3})~)+H(~P_{123}~)+H(~I~)\\
=&t(H(~P~)-(H(~P_{12}~)+H(~P_{3}~)-H(~P_{123}~)))+H(~P_{123}~)+H(~I~)\\
=&t(\frac{1}{(1-t^{2})^{3}}-(\frac{1}{(1-t^{2})(1-t^{4})(1-t^{6})}+\frac{1}{(1-t^{2})^{2}(1-t^{4})}-\frac{1}{(1-t^{4})(1-t^{6})(1-t^{8})}))\\
&+\frac{1}{(1-t^{4})(1-t^{6})(1-t^{8})}+\frac{t^{7}}{(1-t^{2})^{3}}\\
=&\frac{t^{3}+t^{5}+t^{7}-t}{(1-t^{2})(1-t^{4})(1-t^{6})}+\frac{1+t}{(1-t^{4})(1-t^{6})(1-t^{8})}+\frac{t^{7}}{(1-t^{2})^{3}}\\
=&\frac{1+t^5+3t^7+6t^9+7t^{11}+7t^{13}+5t^{15}+3t^{17}+t^{19}}{(1-t^{4})(1-t^{6})(1-t^{8})}.
\end{split}
\end{equation*}
\end{Ex}

\section{An application of the computation of cohomology groups\label{s6}}
As an application of our computation \ref{s3t1}, \ref{s3t2} and \ref{s4t1}, we prove that there is a $p$-torsion  for each prime $p$ in the integral cohomology groups of classifying spaces of rank 3 infinite Kac-Moody groups. We need two lemmas to prove this statement. The first one is a generalization of Borel's Theorem \ref{s2t1} for rank 3 Kac-Moody groups, and this is just a direct corollary of our computation of cohomology groups.
\begin{Lem}\label{s6l1}Let $p$ be any prime. If $\mathbb{F}$ is $\mathbb{Q}$ or $\mathbb{F}_{p}$, then the image of induced
cohomology homomorphism $\rho^{*}_{\mathbb{F}}:H^{*}(BK(A);\,\mathbb{F})\rightarrow H^{*}(BT;\,\mathbb{F})$ is $H^{*}(BT;\,\mathbb{F})^{W(A)}$.
\end{Lem}
\begin{proof}
By Lemma \ref{s3l2}, we have $$\mathrm{im}\,\rho^{*}_{\mathbb{F}} \subseteq H^{*}(BT;\,\mathbb{F})^{W(A)}.$$
From Theorems \ref{s3t1}, \ref{s3t2} and \ref{s4t1}, it follows that $H^{*}(BT;\,\mathbb{F})^{W(A)}=P_{123}$ is a drect summand of $H^{*}(BK(A);\,\mathbb{F})$ and hence
$\mathrm{im}\,\rho^{*}_{\mathbb{F}}= H^{*}(BT;\,\mathbb{F})^{W(A)}.$
\end{proof}

The induced homomorphism $\rho^{*}_{\mathbb{Z}}: H^{*}(BK;\,\mathbb{Z})\rightarrow H^{*}(BT;\,\mathbb{Z})$ must map the torsion part of $H^{*}(BK;\,\mathbb{Z})$ to $0$ since $H^{*}(BT;\,\mathbb{Z})$ is a free abelian group, which induces a homomorphism $\widetilde{\rho}^{*}_{\mathbb{Z}}: H^{*}(BK;\,\mathbb{Z})^{free}\rightarrow H^{*}(BT;\,\mathbb{Z})$. Then, we have two commutative diagrams,
\begin{center}
$\xymatrix{
H^{*}(BK(A);\,\mathbb{Z})^{free} \ar[d]_{\beta} \ar[r]^{\widetilde{\rho}^{*}_{\mathbb{Z}}} &H^{*}(BT;\,\mathbb{Z}) \ar[d]^{\beta} \\
H^{*}(BK(A);\,\mathbb{Q})\ar[r]^{\rho^{*}_{\mathbb{Q}}} & H^{*}(BT;\,\mathbb{Q}),}$
$\xymatrix{
H^{*}(BK(A);\,\mathbb{Z})^{free} \ar[d]_{\gamma} \ar[r]^{\widetilde{\rho}^{*}_{\mathbb{Z}}} &H^{*}(BT;\,\mathbb{Z}) \ar[d]^{\gamma} \\
H^{*}(BK(A);\mathbb{F}_{p})\ar[r]^{\rho^{*}_{\mathbb{F}_{p}}} & H^{*}(BT;\mathbb{F}_{p}),}$
\end{center}
where $\beta$ and $\gamma$ are the cohomology homomorphisms induced, respectively, by coefficient group homomorphisms $\mathbb{Z} \rightarrow \mathbb{Q}$ and $\mathbb{Z} \rightarrow \mathbb{F}_{p}$. Moreover, by the universal coefficients theorem and the condition that $H^{*}(BK(A);\,\mathbb{Z})$ has no a $p$-torsion , we know that $\beta$ and $\gamma$ are surjective for $H^{*}(BK(A);\,\mathbb{Z})^{free}$.
Combining these two commutative diagrams, we have the following lemma.
\begin{Lem}\label{s6l2} If $H^{*}(BK(A);\,\mathbb{Z})$ has no a $p$-torsion , then
$$\mathrm{im}\,\rho^{*}_{\mathbb{Q}}= \mathrm{im}\,\beta\circ\widetilde{\rho}^*_{\mathbb{Z}}\cong\mathrm{im}\,\widetilde{\rho}^*_{\mathbb{Z}}\otimes\mathbb{Q}~,~ \mathrm{im}\,\rho^{*}_{\mathbb{F}_{p}}= \mathrm{im}\,\gamma\circ\widetilde{\rho}^*_{\mathbb{Z}}\cong\mathrm{im}\,\widetilde{\rho}^{*}_{\mathbb{Z}}\otimes\mathbb{F}_{p}.$$
\end{Lem}

The main result of this section is stated as follows.
\begin{Thm}\label{s6t1}
There is a $p$-torsion  for each prime $p$ in the integral cohomology groups of classifying spaces of rank 3 infinite Kac-Moody groups.
\end{Thm}

\begin{proof}
If there is a prime $p$ such that $H^{*}(BK(A);\,\mathbb{Z})$ has no a $p$-torsion , then $H^{*}(BK(A);\,\mathbb{Z})$ consists of a free part and a torsion part that has no $p$-torsion .

It is easy to see that $\widetilde{\rho}^*_{\mathbb{Z}}$ always maps the generator of $H^{*}(BK(A);\,\mathbb{Z})^{free}$ to some
$\mathbb{Z}$-linear combination of generators of $H^{*}(BT;\,\mathbb{Z})$. Certainly, for $\mathrm{im}\,\widetilde{\rho}^*_{\mathbb{Z}}$, we may choose a basis from the image of the generators of $H^{*}(BK(A);\,\mathbb{Z})^{free}$. This basis is also a basis of $\mathrm{im}\,\widetilde{\rho}^*_{\mathbb{Z}}\otimes\mathbb{Q}$ since $\mathbb{Z}$-linear independence is equivalent to $\mathbb{Q}$-linear independence. But when we tensor it with
$\mathbb{F}_{p}$, this may not be a basis of $\mathrm{im}\,\widetilde{\rho}^*_{\mathbb{Z}}\otimes\mathbb{F}_{p}$, thus
$${\rm rank}(~\mathrm{im}\,\widetilde{\rho}^*_{\mathbb{Z}}\otimes\mathbb{Q}~)\geq {\rm rank}(~\mathrm{im}\,\widetilde{\rho}^*_{\mathbb{Z}}\otimes\mathbb{F}_{p}~),$$
which, together with Lemma \ref{s6l2}, implies that
\begin{equation}\label{s6f1}
{\rm rank}(~\mathrm{im}\,\rho^{*}_{\mathbb{Q}}~) \geq {\rm rank}(~\mathrm{im}\,\rho^{*}_{\mathbb{F}_{p}}~).
\end{equation}
In addition, by Lemma \ref{s6l1}, we know that
$$\mathrm{im}\,\rho^{*}_{\mathbb{Q}}\cong H^{*}(BT;\,\mathbb{Q})^{W(A)},~~~\mathrm{im}\,\rho^{*}_{\mathbb{F}_{p}}\cong H^{*}(BT;\,\mathbb{F}_{p})^{W(A)}.$$
Then, the inequality \ref{s6f1} can be rewritten as
\begin{equation}\label{s6f2}
{\rm rank}(~H^{*}(BT;\,\mathbb{Q})^{W(A)}~) \geq {\rm rank}(~H^{*}(BT;\,\mathbb{F}_{p})^{W(A)}~).
\end{equation}
Note that $$H^{*}(BT;\,\mathbb{F}_{p})^{GL(3,\,\mathbb{F}_{p})} \subseteq H^{*}(BT;\,\mathbb{F}_{p})^{W(A)}$$
and
$$H^{*}(BT;\,\mathbb{F}_{p})^{GL(3,\,\mathbb{F}_{p})}\cong\mathbb{F}_{p}[d_{2(p^{3}-p^{2})},~d_{2(p^{3}-p)},~d_{2(p^{3}-1)}],$$
where $GL(3,\,\mathbb{F}_{p})$ is the general linear group of degree 3 over $\mathbb{F}_{p}$ and $d_{2(p^{3}-p^{2})},~d_{2(p^{3}-p)}$, $d_{2(p^{3}-1)}$
are Dickson invariants \cite{Dic11}. Thus, there exists an inequality of Poincar\'e series
\begin{equation}\label{s6f3}
H(~H^{*}(BT;\,\mathbb{F}_{p})^{W(A)};~t~)\geq H(~H^{*}(BT;\,\mathbb{F}_{p})^{GL(3,\,\mathbb{F}_{p})};~t~)=\frac{1}{(1-t^{2(p^{3}-p^{2})})(1-t^{2(p^{3}-p)})(1-t^{2(p^{3}-1)})}.
\end{equation}
In addition, from Theorem \ref{s5t1}, if $A$ is nonsymmetrizable, we deduce that
$$H^{*}(BT;\,\mathbb{Q})^{W(A)}\cong \mathbb{Q}~~{\rm and ~~hence}~~H(~H^{*}(BT;\,\mathbb{Q})^{W(A)};~t~)=1.$$ Therefore, $$H(~H^{*}(BT;\,\mathbb{F}_{p})^{W(A)};~t~)\geq
\frac{1}{(1-t^{2(p^{3}-p^{2})})(1-t^{2(p^{3}-p)})(1-t^{2(p^{3}-1)})} > H(~H^{*}(BT;\,\mathbb{Q})^{W(A)};~t~),$$
which contradicts \ref{s6f2}. On the other hand, if $A$ is symmetrizable, we deduce that $$H^{*}(BT;\,\mathbb{Q})^{W(A)}\cong \mathbb{Q}[\psi]~~{\rm
and~~hence}~~H(~H^{*}(BT;\,\mathbb{Q})^{W(A)};~t~)=\frac{1}{1-t^{4}}.$$
Now, consider the coefficient of $t^{2p^{2}(p^{2}-1)}$ in power series $$H(~H^{*}(BT;\,\mathbb{F}_{p})^{GL(3,\,\mathbb{F}_{p})};~t~)~~{\rm and}~~H(~H^{*}(BT;\,\mathbb{Q})^{W(A)};~t~).$$
Note that for the first power series, the coefficient is at least two, and for the second the coefficient, it is one; we conclude that $${\rm rank} (~H^{2p^{2}(p^{2}-1)}(BT;\,\mathbb{F}_{p})^{GL(3,\,\mathbb{F}_{p})}~)> {\rm rank}(~H^{2p^{2}(p^{2}-1)}(BT;\,\mathbb{Q})^{W(A)}~),$$
which, combined with \ref{s6f3}, implies that
$${\rm rank} (~H^{2p^{2}(p^{2}-1)}(BT;\,\mathbb{F}_{p})^{W(A)}~)> {\rm rank}(~H^{2p^{2}(p^{2}-1)}(BT;\,\mathbb{Q})^{W(A)}~),$$
and this also contradicts \ref{s6f2}. This completes the proof.
\end{proof}

\section{The ring structure of \texorpdfstring{$H^*(BK(A);\,\mathbb{Q})$}{} \label{s8}}

In this section, we try to determine the ring structure of $H^*(BK(A);\,\mathbb{Q})$.

\subsection{Vanishing results for certain direct summand of \texorpdfstring{$H^*(BK(A);\,\mathbb{Q})$}{}}

We have already shown that $H^*(BK(A);\,\mathbb{Q})$ can be represented by a direct sum of the invariants of Weyl groups and their quotients. Then it is natural to ask whether those results could be reduced further according to the given classification. It turns out that some results can be refined, which simplifies the computation of the cup product of $H^*(BK(A);\,\mathbb{Q})$. In particular, for the rational cohomology, the first case in Theorem \ref{s3t1} can be further reduced by the following theorem, but it can not for the mod $p$ cohomology.
\begin{Thm}\label{s8t1}
Let $A=(a_{ij})_{3\times 3}$ be an infinite Cartan matrix and $p$ be any prime. If $A_{12}$ is also infinite, then $P=P_{1}+P_{2}$ holds for rational invariants but not for mod $p$ invariants. Namely,
\begin{align*}
\mathbb{Q}[\omega_{1},~\omega_{2},~\omega_{3}]&= \mathbb{Q}[\omega_{1},~\omega_{2},~\omega_{3}]^{\langle\sigma_{1}\rangle}+ \mathbb{Q}[\omega_{1},~\omega_{2},~\omega_{3}]^{\langle\sigma_{2}\rangle}\\
&= \mathbb{Q}[\omega_{1}(\omega_{1}-\alpha_1),~\omega_{2},~\omega_{3}]+\mathbb{Q}[\omega_{1},~\omega_{2}(\omega_{2}-\alpha_2),~\omega_{3} ],\end{align*}
$${\mathbb{F}_p}[\omega_{1},~\omega_{2},~\omega_{3}]\neq {\mathbb{F}_p}[\omega_{1}(\omega_{1}-\alpha_1),~\omega_{2},~\omega_{3}]+{\mathbb{F}_p}[\omega_{1},~ \omega_{2}(\omega_{2}-\alpha_2),~\omega_{3}].$$
\end{Thm}
\begin{Rem}
Let $\mathbb{F}$ be any field. The action of Weyl group $W(A)$ on $ H^{*}(BT;\,\mathbb{F})$ induces a representation in $GL(3,\,\mathbb{F})$. Recall the condition that $A_{12}$ is infinite, which implies that the Weyl subgroup $W_{12}(A)$ is infinite. The essential difference between the rational case and mod $p$ case is that the image of $W_{12}(A)$ is still infinite in $GL(3,\,\mathbb{Q})$ but finite in $GL(3,\,\mathbb{F}_{p})$.
\end{Rem}
To prove Theorem \ref{s8t1}, we need the following lemma.
\begin{Lem}\label{s8l1}
Let $\kappa$ denote
$a_{12}\omega_{1}^{2}+a_{12}a_{21}\omega_{1}\omega_{2}+a_{21}\omega_{2}^{2}+a_{12}a_{31}\omega_{1}\omega_{3}+a_{21}a_{32}\omega_{2}\omega_{3}$, then
$$P_{12}= \mathbb{Q}[\kappa,~\omega_{3}].$$
\end{Lem}
\begin{Rem}
Note that $\kappa$ has a primitive form denoted by $\kappa'$. By directly checking, we know that $\kappa'$ and $\omega_{3}$ belong to $P_{12}$ for the mod $p$ coefficient. From Theorem \ref{s8t1}, it follows that $ \mathbb{F}_p[\kappa',~\omega_{3}] \subsetneq P_{12}$.
\end{Rem}
\begin{proof}[Proof of Lemma \ref{s8l1}] It is easy to check that $\omega_{3},~\kappa \in P_{12}$, so $\mathbb{Q}[\kappa,~\omega_{3}]\subseteq P_{12}$. Then it suffices to prove that $$P_{12} \subseteq \mathbb{Q}[\kappa,~\omega_{3}],$$ which is equivalent to proving that
any homogeneous polynomial $f\in P_{12} \subset \mathbb{Q}[\omega_{1},~\omega_{2},~\omega_{3}]$ can be expressed by some polynomial of $\omega_{3}$ and $\kappa$. Then, for any homogeneous polynomial $f\in P_{12}$, $f$ can be written uniquely as
$$\sum\limits_{i=0}^{\deg f} h_{i}(\omega_{1},~\omega_{2})\omega_{3}^{i},$$
where $h_{i}(\omega_{1},~\omega_{2})$ is a homogeneous polynomial of $\omega_{1}$ and $\omega_{2}$ for $0\leq i\leq \deg f$. It is easy to see that
$$f(\omega_{1},~\omega_{2},~0)\in \mathbb{Q}[\omega_{1},~\omega_{2}]^{\langle\sigma^{'}_{1},~\sigma^{'}_{2}\rangle},$$
where $\sigma^{'}_{1}$ and $\sigma^{'}_{2}$ are the generators of $W(A_{12})$. Moreover, from Theorem \ref{s6t1} and the fact that every rank two Cartan matrix is symmetrizable, we deduce that $f(\omega_{1},~\omega_{2},~0)= h_{0}(\omega_{1},~\omega_{2})$ is a homogeneous polynomial of the Killing form $$\psi=a_{12}\omega_{1}^{2}+a_{12}a_{21}\omega_{1}\omega_{2}+a_{21}\omega_{2}^{2}.$$
Thus,
\[h_{0}(\omega_{1},~\omega_{2})=\left\{\begin{array}{ll}
0&~~\text{if $\deg f$ is odd},\\
\lambda_{0}\psi^{\frac{\deg f}{2}}\,&
~~\text{if $\deg f$ is even},
\end{array}\right.\]
where $\lambda_{0} \in \mathbb{Q} $ is a constant.

Next, suppose $\deg f = 2n~(n \in \mathbb{N})$,~$f$ can be expressed as
$$f(\omega_{1},~\omega_{2},~\omega_{3})=\lambda_{0}\psi^{n}+\sum\limits_{i=1}^{2n-1} h_{i}(\omega_{1},
~\omega_{2})\omega_{3}^{i}+\lambda_{2n}\omega_{3}^{2n},$$
and we can rewrite this equation as
\begin{equation}\label{s8f1}
f(\omega_{1},~\omega_{2},~\omega_{3})-\lambda_{0}\kappa^{n}=\lambda_{0}(\psi^{n}-\kappa^{n})+\sum\limits_{i=1}^{2n-1} h_{i}(\omega_{1}, ~\omega_{2})\omega_{3}^{i}+\lambda_{2n}\omega_{3}^{2n}=\omega_{3}(\sum\limits_{i=0}^{2n-1} h^{'}_{i}(\omega_{1},~\omega_{2})\omega_{3}^{i}).
\end{equation}
Furthermore, by the fact that $$f(\omega_{1},~\omega_{2},~\omega_{3})-\lambda_{0}\kappa^{n}, \omega_{3} \in P_{12},$$
we have $$\frac{f(\omega_{1},~\omega_{2},~\omega_{3})-\lambda_{0}\kappa^{n}}{\omega_{3}}\in P_{12}.$$
With the same discussion as above, we obtain $h^{'}_{0}(\omega_{1},~\omega_{2})=0$.
Repeating this process until the degree of the right side of \ref{s8f1} decreases to zero, then we have $$f(\omega_{1},~\omega_{2}, ~\omega_{3})=\sum\limits_{i=0}^{n} g_{2i}(\kappa)\omega_{3}^{2i}.$$

On the other hand, for $\deg f$ odd, the argument is similar. Thus, we conclude that $$P_{12}\subseteq \mathbb{Q}[\kappa,\omega_{3}].$$ This finishes the proof of Lemma \ref{s8l1}.
\end{proof}

Now, we prove Theorem \ref{s8t1}.
\begin{proof}[Proof of Theorem \ref{s8t1}] According to the two parts of the result in the theorem, we also divide this proof into two parts: the rational case and mod $p$ case. We first prove the rational case. Let $H(Q)$ denote the Poincar\'e series of $Q$ with indeterminate $t$. From Lemma \ref{s8l1}, it follows that $$H(~P_{12}~)=\frac{1}{(1-t^{2})(1-t^{4})}.$$
Note that
\begin{align*}
H(~P_{1}+ P_{2}~)&= H(~P_{1}~)+ H(~P_{2}~)- H(~P_{12}~)\\
&= \frac{1}{(1-t^{2})^{2}(1-t^{4})}+ \frac{1}{(1-t^{2})^{2}(1-t^{4})}- \frac{1}{(1-t^{2})(1-t^{4})}\\
&= \frac{1+t^{2}}{(1-t^{2})^{2}(1-t^{4})}= \frac{1}{(1-t^{2})^{3}}=H(~P~),
\end{align*}
Therefore, $P=P_{1}+P_{2}$ holds for the rational coefficient, which finishes the first part of this proof.

To complete this proof, we need to prove the second part: the mod $p$ case.
If $P= P_{1}+P_{2}$ holds for the mod $p$ coefficient, we have
\begin{align*}
H(~P_{12}~)&=H(~P_{1}~)+H(~P_{2}~)-H(~P_{1}+P_{2}~)\\
&=H(~P_{1}~)+H(~P_{2}~)-H(~P~)\\
&=\frac{1}{(1-t^{2})^{2}(1-t^{4})}+ \frac{1}{(1-t^{2})^{2}(1-t^{4})}- \frac{1}{(1-t^{2})^3}\\
&=\frac{1}{(1-t^{2})(1-t^{4})}.
\end{align*}
Note that $$H^{*}(BT;\,\mathbb{F}_{p})^{GL(3,\,\mathbb{F}_{p})} \subseteq H^{*}(BT;\,\mathbb{F}_{p})^{W_{12}(A)}$$
and
$$H^{*}(BT;\,\mathbb{F}_{p})^{GL(3,\,\mathbb{F}_{p})}\cong\mathbb{F}_{p}[d_{2(p^{3}-p^{2})},~d_{2(p^{3}-p)},~d_{2(p^{3}-1)}].$$
Thus, there is an inequality of Poincar\'e series
$$H(~H^{*}(BT;\,\mathbb{F}_{p})^{W_{12}(A)};~t~)\geq H(~H^{*}(BT;\,\mathbb{F}_{p})^{GL(3,\,\mathbb{F}_{p})};~t~),$$
namely,$$\frac{1}{(1-t^{2})(1-t^{4})}\geq\frac{1}{(1-t^{2(p^{3}-p^{2})})(1-t^{2(p^{3}-p)})(1-t^{2(p^{3}-1)})}.$$
However, the coefficient of the right side of the series may be greater than that on the left side, which is a contradiction. This finishes the proof of the second part and hence of Theorem \ref{s8t1}.
\end{proof}

By Theorem \ref{s8t1}, we can simplify the results of Theorem \ref{s3t1}.
\begin{Thm}\label{s8t2}
The same notations apply as in Theorem \ref{s3t1}. For rank 3 infinite Kac-Moody groups, the rational cohomology of the first two classes ${\rm (i), (ii)}$ could be reduced to the third class ${\rm (iii)}$.
\end{Thm}
\begin{proof}
First, recall the results of Theorem \ref{s3t1},
\begin{center}
\begin{tabular}{|l|c|}
\hline Homotopy type   &  Cohomology group \\
\hline [{\rm (i)}] $\{\{1\},~\{2\},~\{3\}\}$&$\Sigma(P-(P_{1}+P_{2}))\oplus \Sigma(P-(P_{12}+P_{3}))\oplus P_{123}$ \\
\hline [{\rm (ii)}] $\{\{1, 2\},~\{3\}\}$&$\Sigma(P-(P_{12}+P_{3}))\oplus P_{123}$\\
\hline [{\rm (iii)}]$\{\{1, 2\},~\{1, 3\}\}$&$\Sigma(P_{1}-(P_{12}+P_{13}))\oplus P_{123}$\\
\hline [{\rm (iv)}] $\{\{1, 2\},~\{1, 3\},~\{2, 3\}\}$ & $\Sigma^{2}(P-(P_{1}+P_{2}+P_{3}))\oplus\Sigma(P_{1}\cap(P_{2}+P_{3})-(P_{12}+P_{13}))\oplus P_{123}$.\\
\hline
\end{tabular}
\end{center}
Therefore, via Theorem \ref{s8t1}, we deduce that $\Sigma(P-(P_{1}+P_{2}))$ is zero for the rational cohomology for case (i). Consequently, the rational cohomologies of (i) and (ii) are the same.
A comparison between (ii) and (iii) showed that $\{1, 3\}$ in (iii) becomes infinite in (ii). By Theorem \ref{s8t1}, we have $P=P_{1}+P_{3}$, which implies that $$P-(P_{12}+P_{3})=(P_{1}+(P_{12}+P_{3}))-(P_{12}+P_{3})\cong P_{1}-P_{1}\cap(P_{12}+P_{3}).$$
Since $P_{12}\subset P_{1}$, we deduce that $$P_{1}\cap (P_{12}+P_{3})=P_{1}\cap P_{12}+P_{1}\cap P_{3}=P_{12}+P_{13}.$$
Thus, $$P-(P_{12}+P_{3})\cong P_{1}-(P_{12}+P_{13}).$$
Therefore, the cohomology of (ii) can be reduced to (iii) by Theorem \ref{s8t1}, so does the cohomology of (i).
\end{proof}

\bigskip

Foley once guessed that $\Sigma^{2}(P-(P_{1}+P_{2}+P_{3}))$ is also zero in \cite{Fol22}; if this were true, our results could be reduced further, but we cannot prove it. His guess can be stated generally as follows:
\begin{Conj}\label{Fc}
If $A$ is an infinite $n\times n$ Cartan matrix, then$$ P=P_{1}+P_{2}+\cdot\cdot\cdot+P_{n},$$
where $P=\mathbb{Q}[\omega_{1},~\omega_{2},~\cdots,~\omega_{n}]$ and $P_{i}$ denotes the invariant $H^{*}(BT;\,\mathbb{Q})^{\langle\sigma_{i}\rangle}$ for $\sigma_{i} \in W(A),~1\leq i \leq n$.
\end{Conj}

\subsection{The rational cohomology ring \texorpdfstring{$H^*(BK(A);\,\mathbb{Q})$}{}}

\begin{Thm}\label{RCR}
The same notations apply as in Theorem \ref{s3t1}. For the classes ${\rm (i), (ii), (iii)}$, the rational cohomology ring $H^*(BK(A);\,\mathbb{Q})$ is isomorphic to
\begin{enumerate}
\item $\bigoplus\limits_{i=2}^\infty \bigoplus\limits_{j=1}^{\alpha_{i}} x_{2i+1}^j$ with $H(H^*(BK(A);\,\mathbb{Q}))=1 + \alpha_2 t^5+\alpha_3 t^7 +\cdots +\alpha_{i} t^{2i+1}+\cdots $ when $A$ is nonsymmetrizable;
\item $\mathbb{Q}[\psi]\otimes\bigoplus\limits_{i=2}^\infty \bigoplus\limits_{j=1}^{\beta_{i}} y_{2i+1}^j$ with $(1-t^4)H(H^*(BK(A);\,\mathbb{Q}))=1 + \beta_2 t^5+\beta_3 t^7 +\cdots +\beta_{i} t^{2i+1}+\cdots $when $A$ is symmetrizable.
\end{enumerate}
\end{Thm}
\begin{proof}
See the proof of Theorem \ref{r4t3}.
\end{proof}

\begin{Thm}\label{r4t3}
 If Conjecture \ref{Fc} is true for rank 3 case, the rational cohomology ring $H^*(BK(A);\,\mathbb{Q})$ of the class ${\rm (iv)}$ can be computed as these three classes.
\end{Thm}

\begin{proof}
By Theorem \ref{s8t2}, it suffices to compute the rational cup product of $H^*(BK(A);\,\mathbb{Q})$ for case ${\rm (iii)}$. The push-out diagram of ${\rm (iii)}:  BK(A)\simeq BK_{12}(A)\cup BK_{13}(A)$
$$\CD
BT@>j_1>>BK_{12}(A) \\
  @V j_2 VV @V i_1 VV  \\
  BK_{3}(A) @>i_2>> BK_{12}(A)\cup BK_{3}(A)
\endCD$$
induces the following commutative diagram:
\begin{center}
$\xymatrix{
  H^{*}(BK_{12}(A)\cup BK_{3}(A);\mathbb{Q})\ar[d]_{i_2^*} \ar[r]^{i_1^*\hspace{-0.8cm}} &H^{*}(BK_{3}(A);\mathbb{Q}) \ar[d]^{j_1^*} \\
   H^{*}(BK_{12}(A);\mathbb{Q})\ar[r]^{j_2^*\hspace{-0.8cm}} & H^{*}(BT;\,\mathbb{Q}),}
\xymatrix{
  \Sigma(P_{1}-(P_{12}+P_{13}))\oplus P_{123}\ar[d]_{i_2^*} \ar[r]^{i_1^*\hspace{-1cm}} &P_3 \ar[d]^{j_1^*}\\
   P_{12}\ar[r]^{j_2^*\hspace{-1cm}} & P.}$
\end{center}
By the facts that $i_1^*$ preserves degrees and it is injective when it restricts to $P_{123}$, we obtain that $\ker i_1^*=\Sigma(P_{1}-(P_{12}+P_{13}))$, which means that $\Sigma(P_{1}-(P_{12}+P_{13}))$ is an ideal of $H^*(BK(A);\,\mathbb{Q})$. Obviously, $P_{1},P_{12}$, and $P_{13}$ are $P_{123}$ modules, so is $\Sigma(P_{1}-(P_{12}+P_{13}))$.
Recall Theorem \ref{s5t1},
$$P_{123}\cong \left\{\begin{array}{ll}
\mathbb{Q}\hspace{1cm}  \text{if $A$ is nonsymmetrizable},\\
\mathbb{Q}[\psi]\hspace{0.5cm} \text{if $A$ is symmetrizable}.
\end{array}\right.
$$
If $A$ is nonsymmetrizable, then the cup product on $\Sigma(P_{1}-(P_{12}+P_{13}))$ is trivial. If $A$ is symmetrizable, then the cup product on $\Sigma(P_{1}-(P_{12}+P_{13}))$ is also trivial. The reason is stated as follows.
If there exist two elements $x, y$ in $\Sigma(P_{1}-(P_{12}+P_{13}))$ such that the cup product $x\cup y$ is nontrivial, then $x\cup y$ is in $P_{123}$ for degree reason. Since $i_1^*$ is injective on $P_{123}$, $i_1^*(x\cup y)=i_1^*(x)\cup i_1^*(y)$ is nonzero, which contradicts to the fact that $i_1^*(x)=i_1^*(y)=0$. So the cup product on $\Sigma(P_{1}-(P_{12}+P_{13}))$ must be trivial.
\end{proof}

\bigskip

\newpage
\noindent {\sc Yangyang Ruan

\noindent
Beijing Institute of Mathematical Sciences and Applications\\ No. 544, Hefangkou Village, Huaibei Town, Huairou District, Beijing, 101408 \\P.R.China

\noindent{{\it E-mails:}}
\texttt{ruanyy@amss.ac.cn}

\vspace{0.3cm}
\noindent {\sc Xu-an Zhao

\noindent
Beijing Normal University\\ No.19, Xinjiekouwai St, Haidian District, Beijing, 100875\\P.R.China

\noindent{{\it E-mails:}}
\texttt{zhaoxa@bnu.edu.cn}

\end{document}